\documentclass[11 pt]{amsart}
\usepackage{amsmath}
\usepackage{amsthm}
\usepackage[english]{babel}
\usepackage{amssymb}
\usepackage{graphics}
\usepackage{epsfig}
\usepackage{amscd}
\usepackage{amsfonts}
\usepackage{amsbsy}
\usepackage{float}
\usepackage{caption}
\usepackage{subcaption}
\usepackage{overpic}
\usepackage{dsfont}
\usepackage{titlesec}
\usepackage{xspace}
\usepackage{ulem}
\newcommand{\R}{\mathds{R}}

\newcommand{\N}{\ensuremath{\mathbb{N}}}

\newcommand{\Z}{\ensuremath{\mathbb{Z}}}

\newtheorem {theorem} {Theorem}
\newtheorem {prop} [theorem] {Proposition}

\newtheorem {lemma} [theorem] {Lemma}

\newtheorem {remark} {Remark}

\usepackage{color}

\usepackage{fancyhdr}
\begin{document}

\title[Periodic oscillations in electrostatic actuators under time delayed feedback controller]{Periodic oscillations in electrostatic actuators under time delayed feedback controller}

\author[Pablo Amster, Andr\'es Rivera, John A. Arredondo]
{Pablo Amster, Andr\'es Rivera, John A. Arredondo}

\address{$^1$ Departamento de Matem\'aticas
Universidad de Buenos Aires,
Argentina.}
\address{
$^2$ Departamento de Ciencias Naturales y Matem\'aticas
Pontificia Universidad Javeriana Cali, Facultad de Ingenier\'ia y Ciencias,
Calle 18 No. 118--250 Cali, Colombia.}
\address{
$^3$ Departamento de Matemáticas Fundación Universitaria Konrand Lorenz, Facultad de Ciencias e Ingeniería, Cra 9 bis 62-43, Bogotá-Colombia.}

\email{pamster@dm.edu.ar, amrivera@javerianacali.edu.co, alexander.arredondo@konrandlorenz.edu.co}\emph{}



\date{}
\dedicatory{}
\maketitle

\begin{center}\rule{0.9\textwidth}{0.1mm}
\end{center}
\begin{abstract}
In this paper, we prove the existence of two positive $T$-periodic solutions of an electrostatic actuator modeled by the time-delayed Duffing equation
\[
\ddot{x}(t)+f_{D}(x(t),\dot{x}(t))+ x(t)=1- \dfrac{e \mathcal{V}^{2}(t,x(t),x_{d}(t),\dot{x}(t),\dot{x}_{d}(t))}{x^2(t)}, \qquad x(t)\in\,]0,\infty[
\]
where $\displaystyle{x_{d}(t)=x(t-d)}$ and $\displaystyle{\dot{x}_{d}(t)=\dot{x}(t-d),}$ denote position and velocity feedback respectively, and
\[
\mathcal{V}(t,x(t),x_{d}(t),\dot{x}(t),\dot{x}_{d}(t))=V(t)+g_{1}(x(t)-x_{d}(t))+g_{2}(\dot{x}(t)-\dot{x}_{d}(t)),
\]
 is the feedback voltage with positive input voltage $V(t)\in C(\mathbb{R}/T\Z)$ for $e\in \mathbb{R}^{+}, g_{1},g_{2}\in \mathbb{R}$, $d\in [0,T[$. The damping force $f_{D}(x,\dot{x})$ can be linear, i.e.,  $f_{D}(x,\dot{x}) = c\dot{x}$, $c\in\mathbb{R}^+$ or squeeze film type, i.e., $f_{D}(x,\dot{x}) = \gamma\dot{x}/x^{3}$, $\gamma\in\mathbb{R}^+$. The fundamental tool to prove our result is a local continuation method of periodic solutions from the non-delayed case $(d=0)$. Our approach provides new insights into the delay phenomenon on microelectromechanical systems and can be used to study the dynamics of a large class of delayed Li\'enard equations that govern the motion of several actuators, including the comb-drive finger actuator and the torsional actuator. Some numerical examples are provided to illustrate our results.
\end{abstract}
\begin{center}\rule{0.9\textwidth}{0.1mm}
\end{center}

\section*{1. Introduction}
The Nathanson's actuator is a fundamental theoretical model of a recent technology called Micro-Electro-Mechanical Systems (MEMS), in which two parallel plates\footnote{The plate can have any shape, but it is usual to assume a rectangular shape.} (electrodes) are placed at an initial and positive distance. One plate is stationary and the other is allowed to move. Both electrodes are biased by a voltage $V(t)$ where $t$ is an independent variable related to time. Therefore an electrostatic force emerges $f_{E}=f_{E}(t,x)$ acting on both electrodes and pulling the movable one a distance $x=x(t)$, see 
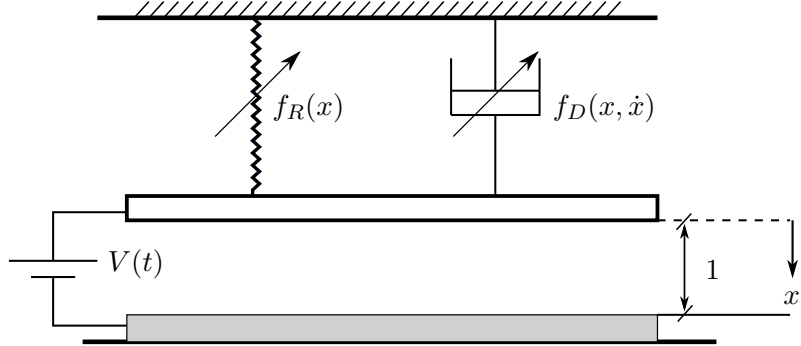
\begin{figure}[t]
\centering
\begin{overpic}[width = 0.7\textwidth, tics = 5]{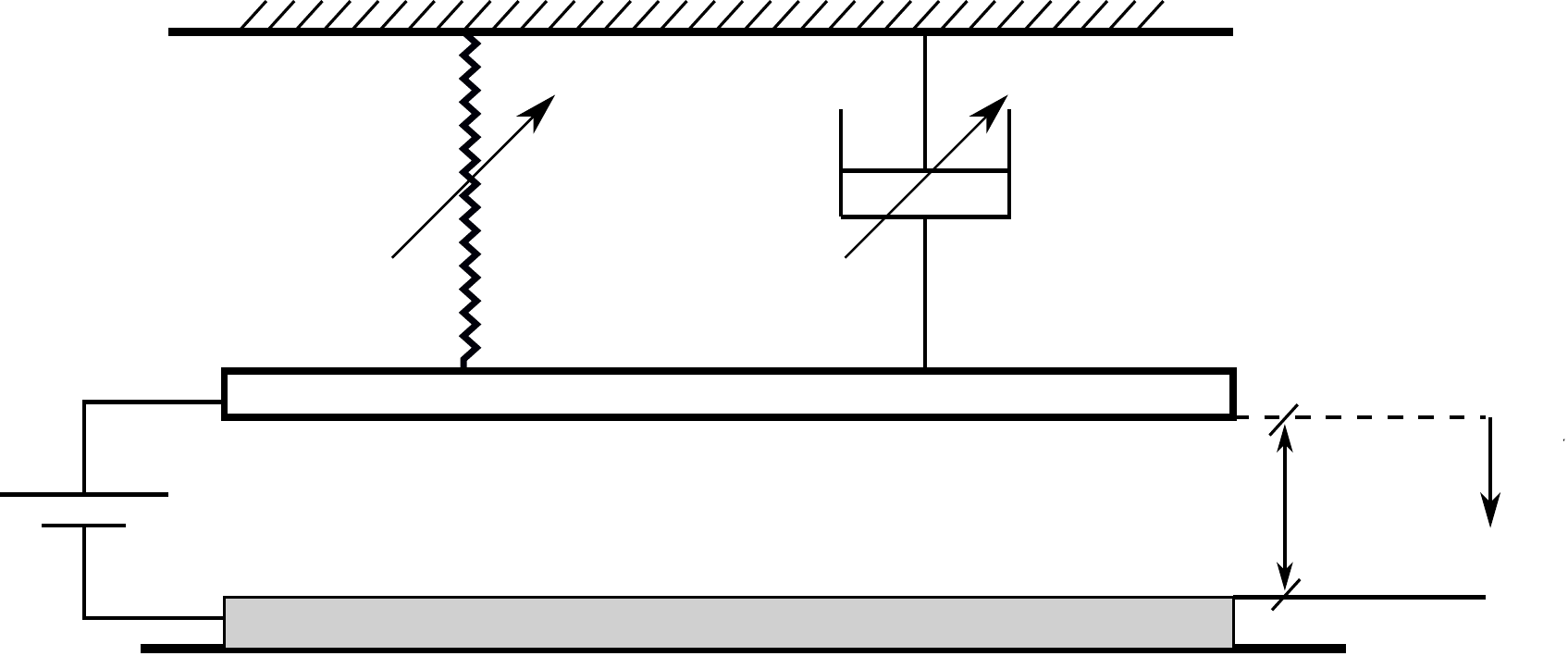}
\put(12,9){$V(t)$}
\put(94,5){$x$}
\put(84.5,8){$1$}
\put(32,28){$f_{R}(x)$}
\put(66,28){$f_{D}(x,\dot{x})$}
\end{overpic}
\caption{Schematic diagram of a parallel plate capacitor with one movable plate, under electrostatic, restoring and damping forces.}
\label{fig:FigNat}
\end{figure}
Figure \ref{fig:FigNat}. Two more forces act over the movable electrode, namely, a restoring force $f_{R}=f_{R}(x)$ and a damping force $f_{D}=f_{D}(x,\dot{x})$ where $\dot{x}=\dot{x}(t)$ is the velocity on the movable electrode. The gravitational force is neglected because it is too small compared with the electrostatic force in micro-structures. 

From the Newton's second law, the acceleration $\ddot{x}=\ddot{x}(t)$ of the movable electrode satisfies the equation
\[
\ddot{x}(t)=f_{E}(t,x(t))+f_{R}(x(t))+f_{D}(x(t),\dot{x}(t)).
\]

In suitable units, the electrostatic force is given by
\[
f_{E}(t,x) = \dfrac{ e\,V^2(t)}{(1-x)^2},
\]
where $e \in \R^{+}$ is a parameter related with the physical and mechanical properties of the actuator such as the facing area of the electrodes and the dielectric constant of the medium between them.
In this document, we assume that the restoring force is given by $f_{R}(x)=-x$. For this kind of device, there are two major types of damping forces \footnote{Other types of damping forces
are known and have been studied, although not as extensively as the mentioned ones. Some examples appear in \cite{ Younis_2011} page 120.} which are 
\begin{align}
    f_{D}(x,\dot{x})&=-c\dot{x}, \hspace{1.7cm}  (\text{linear damping}) \label{eq:linear_damping}\\
    f_{D}(x,\dot{x})&=-\dfrac{\gamma}{(1-x)^3}\dot{x}, \quad (\text{squeeze film damping}) \label{eq:squeeze_film_damping}
\end{align}

with $c,\gamma \in \R^{+}.$ The linear damping force comes from simplifying the problem to a moving sphere in a fluid. Meanwhile, the squeeze-damping force is significantly present in parallel plate actuators that have a proportionally bigger surface area in comparison with the distance between the electrodes. In consequence, the equation of motion of the movable electrode is given by the following Duffing equation
\begin{equation}
\label{eq principal}
\ddot{x}(t) - f_{D}(x(t),\dot{x}(t))+x(t)= \dfrac{e V^2(t)}{(1-x(t))^2}, \qquad x(t)\in]-\infty,1[
\end{equation}
with $f_{D}$ given by \eqref{eq:linear_damping} or \eqref{eq:squeeze_film_damping}. A general review of damping forces in MEMS can be found in \cite{Younis_2011,Zhang_2014} and the references therein. Assuming that $V(t)$ is a  positive, continuous and periodic function and $f_{D}$ is a linear damping force given by \eqref{eq:linear_damping}, the authors in \cite{Llibre_Nunez_Rivera_2018,Gutierrez_Torres_2013} prove analytically the existence of exactly two positive periodic solutions, one stable and the other unstable. Under the same assumptions, recently in \cite{Beron_Rivera} it is shown that the stable periodic solution is actually locally exponentially asymptotically stable with rate of exponential decay $c/2$. In addition, the case when the damping force is given by \eqref{eq:squeeze_film_damping} is also considered in \cite{Beron_Rivera} and the authors prove the existence of at least two periodic solutions, one stable and the other unstable under suitable assumptions.  As in \cite{Beron_Rivera,Nunez_Perdomo_Rivera_2019} throughout this document, we consider a $DC-AC$ voltage $V_{\delta}(t)$ of the form $\displaystyle{V_{\delta}(t)=v_0+\delta v(t),}$ with $v_0 \in \mathbb{R}^{+}$ ($DC$-voltage source), $\delta \in \mathbb{R}_{0}^{+}$ and $v(t)\in C_{T}(\mathbb{R},\mathbb{R})$ with zero average.

\textbf{Delay in MEMS.} Time-delayed MEMS actuators appear in practical control engineering applications. In electrostatic actuators, these time delay phenomena are inherent to the device or generated by design. With the increasing aim of improving the performance of these devices in terms of sensitivity and actuation, their stability properties play a fundamental role, because the global behavior of the device can lead to the undesirable lateral instability effect. Therefore, there is a need for an active control that improves the performance and stability of the actuators. One of the techniques to improve the performance of this type of device is the use of feedback controllers with time delay introduced in \cite{Pyragas_1992}, where it was used to stabilize periodic solutions in chaotic systems. After that, the delayed time in the Pyragas method has been investigated in several cases, as an example,  the works of Younis \textit{et al} 
\cite{Younis-2009, Younis-2010, Younis-2013}  describe the dynamics of a resonant microbeam excited electrically,  the dynamics of delayed feedback MEMS resonators, and the solution for a single degree-freedom resonator model actuated by a $DC$-voltage source and an $AC$-voltage source with a delay feedback controller. For this purpose,   perturbation, multiple scales methods combined with shooting technique and basin of attraction analysis are employed. 
In all the referred works, some analytical support is presented after verifying the results experimentally.

In feedback controllers with time delay, the output signal of this type of controller is a delayed value of the system output from which the current output of the system is subtracted. The delay effect is added to the system by adjusting the voltage signal. Position feedback $x_{d}(t)=x(t-d)$ (see \cite{Younis-2010}) or velocity feedback $\dot{x}_{d}(t)=\dot{x}(t-d)$ (see \cite{Younis-2013}) can be used with the delayed time $d>0$. If we combine those situations, the voltage load is given by
\[
\mathcal{V}(t,\sigma_{1}(x,x_{d},g_{1}),\sigma_{2}(\dot{x},\dot{x}_{d},g_{2}))=V_{\delta}(t)+\sigma_{1}(x,x_{d},g_{1})+\sigma_{2}(\dot{x},\dot{x}_{d},g_{2}),
\]
with
\[
\sigma_{1}(x(t),x_{d}(t),g_{1})=g_{1}(x_{d}(t)-x(t)) \quad \text{and} \quad  \sigma_{2}(\dot{x}(t),\dot{x}_{d}(t),g_{2})=g_{2}(\dot{x}_{d}(t)-\dot{x}(t)).
\]

where $g_{i}\in \R$, $i=1,2$ are the corresponding gains of the controller with respect to the delayed position and delayed velocity respectively. Therefore, under the effects of feedback controllers, the differential equation that governs the motion of the movable electrode in the Nathanson model is given by
\begin{equation}
\label{eq:nathanson_norm feedback}
\ddot{x} -f_{D}(x,\dot{x})+ x = \dfrac{e \mathcal{V}^2(t,\sigma_{1}(x,x_{d},g_1),\sigma_{2}(\dot{x},\dot{x}_{d},g_2))}{(1-x)^2}, \qquad x\in\,]-\infty,1[
\end{equation}

\begin{figure}[t]
\centering
\begin{overpic}[width = 0.6\textwidth, tics = 7]{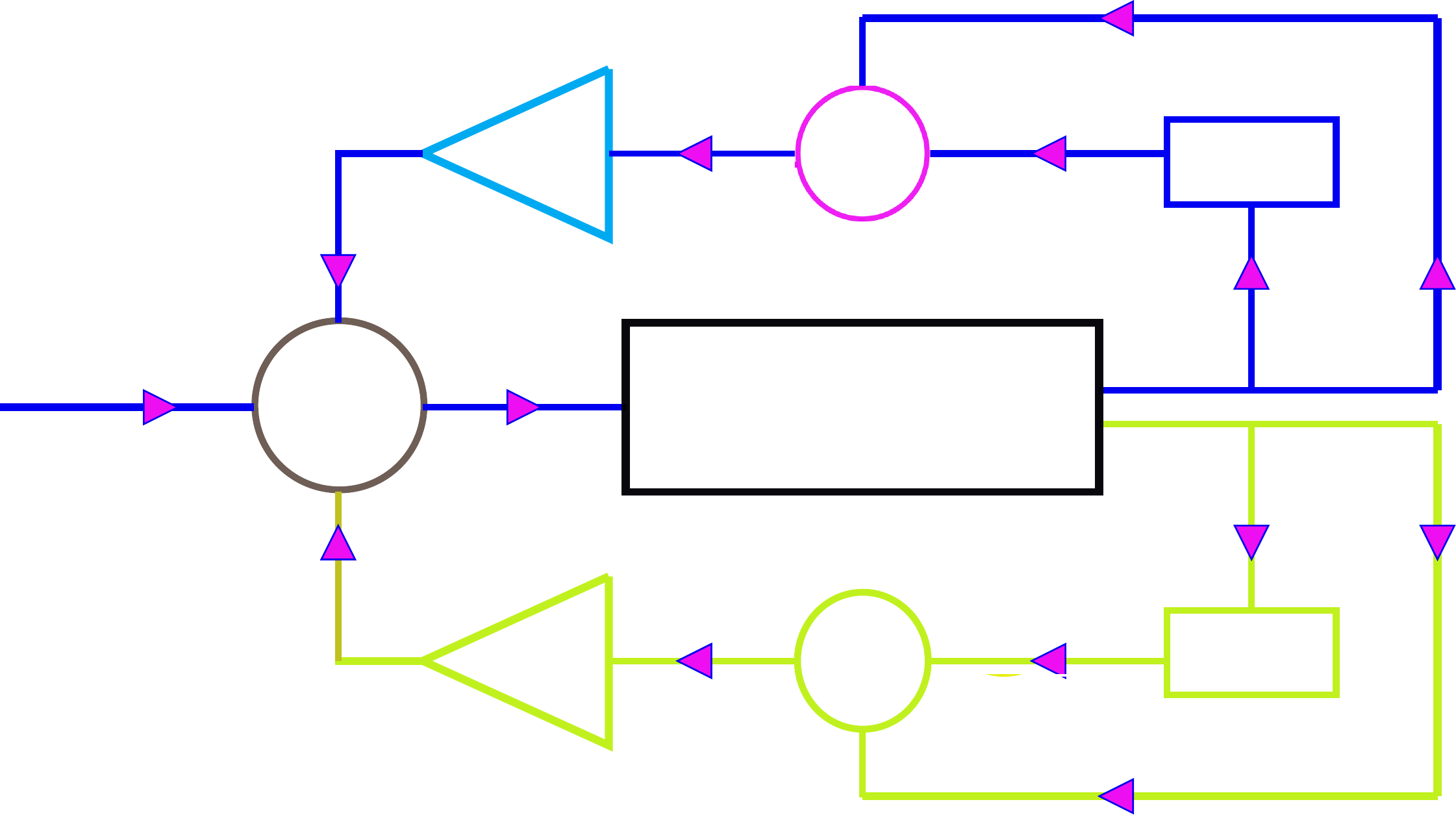}
\put(81.5,43.8){delay}
\put(81.5,10){delay}
\put(47,27){ACTUATOR}
\put(32.2,44){Gain}
\put(57.7,6.5){$-$}
\put(59.5,9.5){$+$}
\put(57.7,47){$-$}
\put(59.5,44.5){$+$}
\put(0.5,31){$V_{\delta}(t)$}
\put(18,27){$+$}
\put(21.5,23.5){$+$}
\put(21.5,30){$+$}
\put(32.2,9){Gain}
\put(65,13){$\dot{x}(t-d)$}
\put(65,41){$x(t-d)$}
\put(82,51){$x(t)$}
\put(82,3){$\dot{x}(t)$}
\end{overpic}
\caption{Diagrammatic representation of two delayed feedback controllers acting over the actuator \eqref{eq principal}. One controller measuring $x(t-d)-x(t)$ and the other measuring $\dot{x}(t-d)-\dot{x}(t).$ }
\label{fig:FigNat delay}
\end{figure}
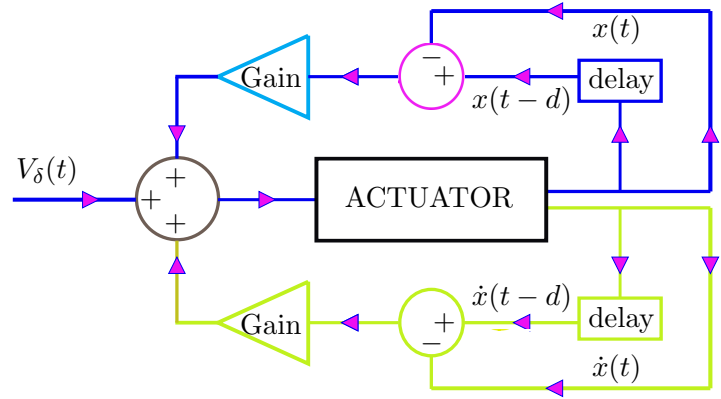

The main goal of this document is to provide an analytical study of the existence of periodic solutions of \eqref{eq:nathanson_norm feedback}. More precisely, we will show that for any value of $d$ and small values of $g_{1}$ and $g_{2}$, there exist two positive $T$-periodic solutions. Moreover, we prove that the same result holds but now under suitable conditions of $g_{1}$ and $g_{2}$ (not necessarily small values) and small values of $d$. In both scenarios, one periodic solution is locally asymptotically stable and the other unstable. The main tools used in this document are the Affinity principle of Kranoselskii and the Implicit Function Theorem in Banach Spaces. After some direct computations, the techniques and ideas in this document can be applied to study periodic motions in other actuators. For example, torsional actuators and comb-drive devices, atomic force microscope micro-cantilevers (see \cite{Younis_2011,AFM}). Let us remark that, despite the large amount of work devoted to MEMS with delay, as far as we know, this is the first document that presents analytical and numerical results of periodic motions in MEMS under the presence of feedback controllers.  The rest of the paper is organized into three sections. In Section 2 we introduce some recent results about the existence, multiplicity, and stability of periodic solutions for \eqref{eq principal} which corresponds to \eqref{eq:nathanson_norm feedback} with $d=0$. In Section 3 we obtain the main results of this document. Firstly, we provide an explicit interval of delay parameters for the local stability of one of the equilibrium points in the delayed autonomous case ($\delta=0$).  Secondly, we state and prove two results about the existence of periodic solutions of \eqref{eq:nathanson_norm feedback}. Finally, in Section 4 we perform some numerical validations of the analytical results presented in Section 3.  

\section*{2. Preliminary results}

From now on and without loss of generality, we move the singularity $x=1$ in \eqref{eq:nathanson_norm feedback} to $x=0$ performing the change of variable $x\to 1-x $. After a slight abuse of notation, the equivalent equation can be written in the form
\begin{equation}
\label{eq:nathanson_norm feedback 2}
\ddot{x}+h_{D}(x,\dot{x})+ x =1- \dfrac{e \mathcal{V}^2(t,s_{1}(x,x_{d},g_1),s_{2}(\dot{x},\dot{x}_{d},g_2))}{x^2}, \qquad x\in\,]0,\infty[
\end{equation}
where
\[
h_{D}(x,\dot{x})=f_{D}(1-x,-\dot{x}), \quad \text{and} \quad s_{1}(x,x_{d},g_1)=g_{1}(x-x_{d}), \quad s_{2}(\dot{x},\dot{x}_{d},g_2)=g_{2}(\dot{x}-\dot{x}_{d}).
\]

Let $\displaystyle{X:]0,\infty[\times \mathbb{R} \to \mathbb{R}^2, t\to X(t)=(x(t),\dot{x}(t))^{tr}}$. Then, $X(t)$ satisfies the non-autonomous delayed system
\begin{equation}
\label{system nathanson_norm feedback}
\dot{X}(t)=F(t,X(t),X_{d}(t)), \qquad F(t,X,X_{d})=\begin{pmatrix}
\dot{x}\\
1-\frac{e \mathcal{V}^2(t,s_{1}(x,x_{d},g_1),s_{2}(\dot{x},\dot{x}_{d},g_2))}{x^2}-x-h_{D}(x,\dot{x})
\end{pmatrix}
\end{equation}
where $X_{d}(t)=X(t-d)$ and $F:\mathbb{R}\times \Omega \times \Omega \to \R^2$, with $\displaystyle{\Omega=\left\{(x,\dot{x})\in \mathbb{R}^2: x>0\right\}.}$
\subsection*{Non-feedback controller} As a fundamental part of this document, we start our study by recalling some known results about the dynamics of the system
\begin{equation}
\label{system nathanson_norm non-feedback}
\dot{X}(t)=F(t,X(t)), \qquad F(t,X)=\begin{pmatrix}
\dot{x}\\
1-\frac{e V_{\delta}^2(t)}{x^2}-x-h_{D}(x,\dot{x})
\end{pmatrix},
\end{equation}
that corresponds to the Nathanson model without feedback controllers. Let's start our analysis assuming that the voltage load is constant, i.e. $\delta=0$, then $V_{0}(t)=v_{0}$ for all $t\in \mathbb{R}$ ($DC$-voltage), therefore we have the autonomous system
\begin{equation}\label{autonomous system}
\dot{X}(t)=F(X), \qquad F(X)=\begin{pmatrix}
\dot{x}\\
1-\frac{e v^2_{0}}{x^2}-x-h_{D}(x,\dot{x})
\end{pmatrix}
\end{equation}
The equilibrium points of \eqref{autonomous system}  are given by  $(x_{\ast},0)$ where $x_{\ast}$ is any solution of the the nonlinear equation
\begin{equation}\label{equilibrium points}
h_{D}(x,0)+x=1-\dfrac{e v_0^2}{x^2}, \quad \Leftrightarrow \quad (1-x)x^{2}=e v_{0}^{2}.
\end{equation}
From the results in \cite{Beron_Rivera,Nunez_Perdomo_Rivera_2019} direct computations prove the following proposition.
\begin{prop} \label{Prop 1} Assume that $0<v_{0}< \frac{2}{9}\sqrt{\frac{3}{e}}$. Then the system \eqref{autonomous system} admits exactly two equilibrium points $(x_{1},0)$, $(x_{2},0)$ such that  
\begin{equation}
0<x_{1}<2/3<x_{2}<1.
\end{equation}
Moreover, the linear system at the equilibrium $(x_{\ast},0)$ with $x_{\ast} \in \left\{x_{1},x_{2}\right\}$ is given by
\begin{equation}\label{linear system}
\dot{X}(t)=DF(X_{\ast})X, \qquad DF(X_{\ast})=\begin{pmatrix}
0 & 1\\
\frac{2-3x_{\ast}}{x_{\ast}} & -\partial_{\dot{x}}h_{D}(x_{\ast},0)
\end{pmatrix}
\end{equation}

Therefore, the equilibrium point $(x_{1},0)$ is always a saddle equilibrium point, meanwhile $(x_{2},0)$ is not a center (i.e., there are not a pair of pure imaginary complex conjugate eigenvalues of $DF(x_{2},0)$. Moreover, $(x_{2},0)$ is a stable spiral equilibrium point if 
\[
\begin{split}
\frac{c^2}{4}&<\frac{3x_{2}-2}{x_2} \quad \text{if $f_{D}(x,\dot{x})$ is a linear damping given by \eqref{eq:linear_damping}}\\
\frac{\gamma^2}{4}&<(3x_{2}-2)x^5_{2} \quad \text{if $f_{D}(x,\dot{x})$ is a squeeze film damping given by \eqref{eq:squeeze_film_damping}}
\end{split}
\]
\end{prop}

\begin{remark}
The threshold value $v^{\ast}_{0}=\frac{2}{9}\sqrt{\frac{3}{e}}$ is called ``pull-in voltage'' and corresponds to the theoretical voltage value when the electrostatic force overcomes a restoring force, inducing the collapse of the two electrodes. In this situation, the device reaches a structural instability phenomenon known as the \textit{pull-in} instability effect. See \cite{Younis_2011,Zhang_2014}
\end{remark}

Concerning the existence of periodic solutions of \eqref{system nathanson_norm non-feedback} we refer to the results in \cite{Gutierrez_Torres_2013, Nunez_Perdomo_Rivera_2019,Beron_Rivera} in which the authors use the method of lower and upper solutions for second order differential equations. Here we  recall the main result in \cite{Beron_Rivera}
\begin{theorem}\label{Periodic solutions existence}
Let us consider the system \eqref{system nathanson_norm non-feedback} and assume that $V_{\delta}=V_{\delta}(t)$ satisfies the condition
\[
0 < V^{2}_{\delta,\min}<V^{2}_{\delta,\max}\leq (v^{\ast}_{0})^{2}=\dfrac{4}{27e}.
\]
where $V^{2}_{\delta,\min}$ and $V^{2}_{\delta,\max}$ denote the minimum and maximum value of $V^{2}_{\delta}$. Let $\xi_{i},\eta_{i}$, $i=1,2$ be constant upper and lower  solutions in $]0,1[$ of \eqref{system nathanson_norm non-feedback} given as solutions of 
\[
(1-x)x^{2}=e V^{2}_{\delta,\max} \quad \text{and} \quad (1-x)x^{2}=e V^{2}_{\delta,\min},
\]
respectively, which satisfy the inequalities
\[
0<\eta_{1}\leq \xi_{1}<2/3<\xi_{2}\leq \eta_{2}<1.
\]
In consequence, 
\begin{enumerate}
    \item In the case $h_{D}(x,\dot{x})=c\dot{x}$ we have:
\begin{itemize}
\item[$\triangleright$] There exists a $T$-periodic solution $\displaystyle{\Psi_{1}(t)=\begin{pmatrix}
\psi_{1}(t)\\
\dot{\psi}_{1}(t)
\end{pmatrix}}$ of \eqref{system nathanson_norm non-feedback} such that 
\[
\eta_{1}<\psi_{1}(t)<\xi_{1}, \quad \forall t \in \mathbb{R}.
\]
\item[$\triangleright$] If $1-2e V_{\delta}^{2}(t)<\Big(\dfrac{\pi}{T}\Big)^{2}+\dfrac{c^2}{4}$ for all $t\in\, [0,T]$,  then there exists a $T$-periodic solution
$\displaystyle{\Psi_{2}(t)=\begin{pmatrix}
\psi_{2}(t)\\
\dot{\psi}_{2}(t)
\end{pmatrix}}$ of \eqref{system nathanson_norm non-feedback} such that
\[
\xi_{2}<\psi_{2}(t)<\eta_{2}, \quad \forall t \in \mathbb{R}.
\]
Furthermore, the periodic solution $\Psi_{1}(t)$ is unstable, meanwhile the periodic solution $\Psi_{2}(t)$ is locally stable. Moreover, if $\displaystyle{c^2/4<(3\xi_{2}-2)/\xi_{2}}$ then $\Psi_{2}(t)$ is locally exponentially asymptotically stable with rate of exponential decay $c/2$. Finally, the only $T$-periodic solutions of \eqref{system nathanson_norm non-feedback} with positive first components are precisely $\Psi_{1}(t)$ and $\Psi_{2}(t)$.
\end{itemize}
\vspace{0.5 cm}

\item In the case $h_{D}(x,\dot{x})=\dfrac{\gamma}{x^3}\dot{x}$ we have:
\begin{itemize}
\item[$\triangleright$] There exists a $T$-periodic solution $\Phi_{1}(t)=\begin{pmatrix}
\phi_{1}(t)\\
\dot{\phi}_{1}(t)
\end{pmatrix}$ of \eqref{system nathanson_norm non-feedback} such that
\[
\eta_{1}<\phi_{1}(t)<\xi_{1}, \quad \forall t \in \mathbb{R}.
\]

\item[$\triangleright$] Define the parameters
\[
N=\dfrac{\gamma}{\eta^{3}_{2}}, \quad \hat{a}=\max\left\{N,\eta_{2}-\xi_{2}\right\},\quad \text{and} \quad M=\dfrac{3\gamma R}{\eta^{4}_{2}}+\dfrac{3\eta_{2}-2}{\eta_{2}},
\]
where $R$ the unique positive solution of $\displaystyle{R-\ln(R+1)=\hat{a}(\eta_{2}-\xi_{2}).}$ Assume that the following conditions hold
\[
M \leq (\pi/T)^2 \quad \text{and} \quad N \leq H(L_*) = \dfrac{L_* - M}{\sqrt{L_*}}\cot \left(\dfrac{T\sqrt{L_*}}{2}\right),
\]
with $L_* \in [M, (\pi/T)^2]$ the solution of
\[
\sin(T\sqrt{L_*}) = T\sqrt{L_*}\left(\dfrac{L_*-M}{L_*+M}\right).
\]
Then, there exists a $T$-periodic solution $\displaystyle{\Phi_{2}(t)=\begin{pmatrix}
\phi_{2}(t)\\
\dot{\phi}_{2}(t)
\end{pmatrix}}$ of \eqref{system nathanson_norm non-feedback} such that
\[
\xi_{2}<\phi_{2}(t)<\eta_{2}, \quad \forall t \in \mathbb{R}.
\]

\end{itemize}
\end{enumerate}
\end{theorem}

\begin{remark}\label{non nontrivial periodic sol}
An important consequence of the condition $\displaystyle{1-eV_{\delta}^{2}(t)<\Big(\dfrac{\pi}{T}\Big)^{2}+\dfrac{c^2}{4}}, \, \forall t\in \mathbb{R}$ is that the linearized system at $\Psi_{1}(t)$ given by
\[
\dot{X}(t)=\begin{pmatrix}
0& 1 \\
\frac{2eV^{2}(t)}{\psi^{3}_{1}(t)}-1& -c
\end{pmatrix}X(t),
\]
has no nontrivial $T$-periodic solutions. The same conclusion is true for $\Psi_{2}(t)$. The proof of this statement can be deduced from the results obtained in \cite{Fitzpatrick}
\end{remark}

\subsection*{Affinity principle of Krasnoselskii} 
Consider the delayed system
\begin{equation}\label{delayed system}
\dot{X}(t)=F(X(t),X(t-d)),
\end{equation}
where $F:\overline{\Omega}\times \overline{\Omega}\to \mathbb{R}^{m}$ is a continuously differentiable function. We shall assume that the point $X_{\ast}\in \Omega$ is an equilibrium point of  \eqref{delayed system}, i.e., $F(X_{\ast},X_{\ast})=0$. Let
\[
A=D_{X}F(X_{\ast},X_{\ast}) \quad \text{and} \quad B=D_{Y}F(X_{\ast},X_{\ast}),
\]
and consider the linear delayed system 
\begin{equation}\label{linear delayed system}
\dot{X}(t)=AX(t)+BX(t-d),
\end{equation} 
corresponding to the linearisation at the equilibrium $X_{\ast}$ of the system 
\eqref{delayed system}. 
It is worth noticing that the Poincar\'e map associated to system \eqref{delayed system} is defined over the infinite-dimensional Banach space $C[-d,0]$, which makes 
the computation of its (Leray-Schauder) index is more difficult than in the non-delayed case. 
Observe, in the first place, that the compactness of $P$ holds only if $d\le T$; however, this assumption is not enough to make the computation trivial. 
The effort is considerably smaller when taking into account a simple version of the Krasnoselskii affinity principle, adapted to our context from \cite{Amster-Epstein}. 
To this end, let $C_T:=C_T(\mathbb R, \mathbb R^m)$ be the space of continuous $T$-periodic
define the (compact) linear operator  $K:C_T\to C_T$ given by $KY:=X$, where $X\in C_T$ is the unique solution of the problem $X'(t) - X(t)= (A-I)Y(t) + BY(t-d)$. 
Roughly speaking, it is seen that the Leray-Schauder degree of $I-P$ over an arbitrary ball $B_r(0)\subset C[-d,0]$
coincides up to a sign with the degree of $I-K$ over $B_r(0)\subset C_T$.

\begin{theorem}[Affinity principle of Krasnoselskii] Let $d\le T$ and assume 
that \eqref{linear delayed system} has no non-trivial $T$-periodic solutions. Then 
$$\deg(I-K,B_r(0),0)= (-1)^m\deg(I-P,B_r(0),0).
$$

\end{theorem}

As a corollary, the latter degree is obtained straightforwardly from the computations in \cite{Amster-Kuna}.

\begin{prop}\label{grado-poinc}
Assume that \eqref{linear delayed system} has no non-trivial $T$-periodic solutions $(T\geq d)$. Then the Poincaré operator 
$P:C[-d,0]\to C[-d,0]$ satisfies
\[
\deg(I-P,B_{r}(0),0)=(-1)^m sgn(\det (A+B)).
\]
\end{prop}

\begin{remark} \label{countable} It is worth point out the following:
   \begin{enumerate}
       \item In order to prove the existence of $T$-periodic solutions, the condition $d\le T$ may be always assumed. This is simply because $X$ is a $T$-periodic solution, for $d$, then it is trivially a $T$-periodic solution for $d+nT$, where $n \in\mathbb Z$ is arbitrary. However, it is worth mentioning that, although the $T$-periodic orbits coincide, the (semi) dynamic of the system changes for different values of $n$. 
       \item The assumption of Proposition \ref{grado-poinc} is satisfied for all $d$ except for a countable set 
       $$N=N_{A,B}\subset [0,T].$$  More precisely, it is readily verified that 
       \eqref{linear delayed system} has no non-trivial $T$-periodic solutions if and only 
       if $h_n(d)\ne 0$ for all $n\in \mathbb Z$, where 
       $$h_n(d):= \det\left(\frac {2n\pi i}T I - A -  e^{-\frac {2n\pi i d}T} B \right).$$
       Because $h_n$ is analytic, the set of its zeros is at most countable for each $n$, and the claim follows. 
   \end{enumerate}
\end{remark}

\begin{prop}\label{periodic solutions}
Let $M\in \mathbb{R}^{m\times m}$ and consider the linear system
\begin{equation}
\label{Floquet System}
\dot{X}=MX.
\end{equation}
Then, the following statements are equivalent
\begin{itemize}
  \item There exists a nontrivial $T$- periodic solution of \eqref{Floquet System}.
    \item $\frac{2k_{0}\pi i}{T}$ is an eigenvalue of $M$ for some $k_{0}\in \N.$
    \item $1$ is a Floquet multiplier of \eqref{Floquet System}.
    \end{itemize}
\end{prop}
\begin{lemma}\label{small delay} Assume that $1$ is not a Floquet multiplier of the linear system
\[
\dot{X}=(A+B)X.
\]
Then, the linear delayed system \eqref{linear delayed system} has no nontrivial $T$-periodic solutions, provided that $d$ is small.
\end{lemma}

\section*{3. Main results}

\subsection*{The delayed autonomous case} Let us consider the system
\begin{equation}
\label{eq:delayed system}
\begin{split}
\dot{X}&=\tilde{F}(X,X_{d}),\\ 
\tilde{F}(X,X_{d})&=\begin{pmatrix}
\dot{x} \\
1-\frac{e(v_{0}+G_{1}s_{1}(x,x_{d},g_{1})+G_{2}s_{2}(\dot{x},\dot{x}_{d},g_{2}))^2}{x^2}-x-h_{D}(x,\dot{x})
\end{pmatrix}
\end{split}
\end{equation}
with $\tilde{F}:\Omega \times \Omega \to \mathbb{R}^2$, which corresponds to the  delayed autonomous system associated to (\ref{delayed system}). Our first result provides an interval of positive values of $d$ such that the local asymptotic stability of the equilibrium point $(x_{2},0)$ of \eqref{eq:delayed system} is maintained. We shall need the following preliminary result.

\begin{prop}
Let $X_{\ast}$ be an equilibrium point of (\ref{eq:delayed system}), i.e. $\tilde{F}(X_{\ast},X_{\ast})=0.$ Then, for $d$ small the linearized delayed system at $X_{\ast}$ has no nontrivial $T$-periodic solutions.
\end{prop} 
\begin{proof}
The equilibrium points of \eqref{eq:delayed system} in $\Omega$ are given by $X_{\ast}=(x_{\ast},0)$ with $x_{\ast}=x_{i}, \,i=1,2$, where $x_{i}$ is a solution of \eqref{equilibrium points}. In consequence, the linear delayed system at the equilibrium $X_{\ast}$ is given as by
\begin{equation}\label{linear delayed}
\dot{X}(t)=\tilde{A}X(t)+\tilde{B}X(t-d),
\end{equation}
with
\begin{equation*}\label{matrices A y B}
\begin{split}
\tilde{A}&=D_{X}\tilde{F}(X_{\ast},X_{\ast})=\begin{pmatrix}
0& 1 \\
\frac{2-3x_{\ast}}{x_{\ast}}-\frac{2ev_{0} g_1}{x^2_{\ast}}&  -\partial_{\dot{x}}h_{D}(x_{\ast},0)-\frac{2ev_{0} g_2}{x^2_{\ast}}
\end{pmatrix},\\
\\
\tilde{B}&=D_{X_{d}}\tilde{F}(X_{\ast},X_{\ast})=\begin{pmatrix}
0& 0 \\
\frac{2ev_{0} g_1}{x^2_{\ast}}& \frac{2ev_{0} g_2}{x^2_{\ast}}
\end{pmatrix}.
\end{split}
\end{equation*}
From Proposition \ref{Prop 1} and Proposition \ref{periodic solutions} the linear system
\[
\dot{X}=(\tilde{A}+\tilde{B})X,
\]
has no nontrivial $T$-periodic solutions because $\frac{2k\pi}{T}i$ is not an eigenvalue of $\tilde{A}+\tilde{B}$ for all $k\in \N.$ A direct application of Lemma  \ref{small delay} concludes the proof.
\end{proof}

Now we are able to provide an explicit interval of delay parameters $d$ such that the local asymptotic stability of the equilibrium $(x_{2},0)$ of (\ref{eq:delayed system}) is guaranteed.
\begin{theorem} \label{delay d0} Let $X_{\ast}=(x_{2},0)$ be the equilibrium point of  (\ref{eq:delayed system}) where $2/3<x_{2}<1$. Define
\[
\begin{split}
&a=\frac{2-3x_{2}}{x_{2}}, \qquad b=-\partial_{\dot{x}}f_{D}(x_{2},0), \\
& \lambda=-\frac{a^2+b^2+1-\sqrt{(a^2+b^2+1)^2-4a^2}}{2a}, \qquad  \hat{g}_{i}=\frac{2(1-x_{2})g_{i}}{v_{0}}, \,\, i=1,2, 
\end{split}
\]
Then $X_{\ast}=(x_{2},0)$ is locally asymptotically stable equilibrium point of \eqref{eq:delayed system} for all $0\leq d < d_{0}$  with
\begin{equation}\label{d0}
d_{0}=\frac{\sqrt{\lambda}|a|}{|\hat{g}_{1}+\hat{g}_{2}|\max\left\{1,\frac{a-1}{b}\right\}\big(\max\left\{1,|a-\hat{g}_{1}|+|b-\hat{g}_{2}|\right\}+|\hat{g}_{1}+\hat{g}_{2}|\big)}.
\end{equation}
\end{theorem}

\begin{proof}
The linearized delayed system of \eqref{eq:delayed system} at $X_{\ast}=(x_{2},0)$ is given by 
\begin{equation}\label{sistema linealizado}
\dot{X}(t)=AX(t)+BX(t-d),
\end{equation}
with
\[
A=\begin{pmatrix}
0& 1 \\
a-\hat{g}_{1}&  b-\hat{g}_{2}
\end{pmatrix} \quad \text{and} \quad B=\begin{pmatrix}
0& 0 \\
\hat{g}_{1}& \hat{g}_{2}
\end{pmatrix}.
\]
From Proposition \ref{Prop 1}, the trivial solution of 
\[
\dot{X}=(A+B)X,
\]
is asymptotically stable. Now, define the matrix $C=(c_{ij})_{1\leq i,j \leq 2}$  with
\begin{equation*}\label{coeficientes de C}
c_{11}=\frac{b^2-a(1-a)}{2ab}, \quad c_{22}=\frac{1-a}{2ab}, \quad \text{and} \quad  c_{12}=c_{21}=-\frac{1}{2a}. 
\end{equation*}

Since $a<0$ and $b<0$, a direct computation proves that $c_{11}>0$ and $\det C>0$. Therefore, by  Sylvester's criterion, the matrix $C$ is a real symmetric positive definite matrix. Moreover, $C$ satisfies the equation
\[
(A+B)^{T}C+C(A+B)=\begin{pmatrix}
2ac_{12} & c_{11}+bc_{12}+ac_{22}\\
c_{11}+bc_{12}+ac_{22} & 2(c_{12}+bc_{22})
\end{pmatrix}=-I_{2}.
\]
The previous conditions correspond precisely to the hypothesis in Theorem \ref{delay size} (see Appendix). In consequence, the trivial solution of \eqref{sistema linealizado} is asymptotically stable for all $0\leq d<d_{0}$ with 
\[
d_{0}=\Big(2\big(||A||+||B||\big)||CB||\Big)^{-1}\Big(\lambda_{\min}(C)/\lambda_{\max}(C)\Big)^{1/2}, 
\]
where $\lambda_{\min}(C)$ and $\lambda_{\max}(C)$  respectively denote the smallest and largest eigenvalues of $C$. Taking the matrix norm  $||P||_{\infty}=\max\left\{|p_{11}|+|p_{12}|, |p_{21}|+|p_{22}|\right\}$ for a given real matrix $p=(p_{i,j})_{1\leq i,j \leq 2}$, direct computations show that
\[
\begin{split}
||A||_{\infty}&=\max\left\{1,|a-\hat{g}_{1}|+|b-\hat{g}_{2}|\right\},\\
||B||_{\infty}&=|\hat{g}_{1}+\hat{g}_{2}|,\\
||CB||_{\infty}&=|\hat{g}_{1}+\hat{g}_{2}|\max\left\{c_{12},c_{22}\right\}=\frac{|\hat{g}_{1}+\hat{g}_{2}|}{2|a|}\max\left\{1,\frac{a-1}{b}\right\},\\
\frac{\lambda_{\min}(C)}{\lambda_{\max}(C)}&=-\frac{a^2+b^2+1-\sqrt{(a^2+b^2+1)^2-4a^2}}{2a}.
\end{split}
\]
To sum up, the equilibrium point $X_{\ast}=(x_{2},0)$ of (\ref{eq:delayed system}) is locally asymptotically stable for all $0\leq d<d_{0}$ with $d_{0}$ given by \eqref{d0}. This completes the proof.
\end{proof}

\subsection*{Periodic solutions under delay effects} In this part of the document, we prove the existence of $T$-periodic solutions of the system \eqref{system nathanson_norm non-feedback} in the case $h_{D}(x,\dot{x})=c\dot{x}$, as a local continuation of the two nontrivial periodic solutions for the non-delayed case. Firstly, we present a continuation theorem using the gain values $g_1$ and $g_{2}$ as continuation parameters.

\begin{theorem}
    Assume that $0<v_0<v_0^*$ and let $x_1$, $x_2$ be the values given by Proposition \ref{Prop 1}. Set $N=N_1\cup N_2$, where $N_{j}$, is as in the Remark \ref{countable} for the respective linearizations at $x_j,$ $j=1,2.$ If $d\in [0,T[\,\backslash N$,
then there exists $\tilde{\delta}>0$ such that 
    problem \eqref{system nathanson_norm feedback} has at least two positive $T$-periodic solutions $X_1$, 
    $X_2$ with $X_1$ unstable. 
\end{theorem}

\begin{proof} 
    Observe, in the first place, that the values $x_j$, $j=1,2$ are also equilibrium points for the delayed model. 
    Furthermore,  
    $A+B$ is the matrix given in    \eqref{linear system}, so by Proposition \ref{grado-poinc} 
    the Leray-Schauder degree of $I-P_L$ for $X_j:=\left(\begin{array}{c}
         x_j \\
         0
    \end{array}\right)$ over $B_r(0)$ is equal to the sign  of $3x_j -2$, $j=1,2$. Next, identify $\mathbb{R}^2$ with the subspace of constant functions of $C[-d,0]$.
    Because the degree is locally constant, taking $r$ and $\delta$ small enough, it follows that 
    $\deg(I-P,B_r(X_j),0)\ne 0$, so the existence of $T$-periodic solutions follows. 
    Furthermore, when $x_j=x_1$ the degree is equal to $-1$, which yields the instability of the $T$-periodic solution. (see \cite{Ortega_1990})
     \end{proof}

\begin{theorem}[Continuation of periodic solutions over the gains]\label{main-theo 1}
Let $\Psi_{i}(t)$, $i\in \left\{1,2\right\}$ be a $T$-periodic solution of \eqref{system nathanson_norm non-feedback}  given in Theorem \ref{Periodic solutions existence} for the case $h_{D}(x,\dot{x})=c\dot{x}$. Then, for each fixed $d\in \mathbb{R}$ there exists a neighborhood $\mathcal{V} \subset \mathbb{R}^{2}$ of $(0,0)$ and a unique function $X^{G}(t)=X(t,d,G)$, where $G=(g_{1},g_{2})$, that is a $T$-periodic solution of \eqref{system nathanson_norm feedback} for all $G \in \mathcal{V}$, i.e., 
{\renewcommand{\labelitemi}{$\triangleright$}
\begin{itemize}
    \item $\dot{X}^{G}(t)=F(t,X^{G}(t),X^{G}_{d}(t),G), \quad X^{G}(t+T)=X^{G}(t),$ for all $t,d\in \mathbb{R}$, and $G \in \mathcal{V}.$
    \item $X^{0}(t)=\Psi_{i}$ for all $t,d\in \mathbb{R}$,with $\Psi_{i}\in \left\{\Psi_{1},\Psi_{2}\right\}$, that is, for $G=(0,0)$.
\end{itemize}}
\end{theorem}

\begin{proof}
Let us consider the Banach spaces
\[
C^{1}_{T}=\left\{X\in C^{1}(\mathbb{R},\mathbb{R}^{2}): X(t+T)=X(t)\right\} \quad \text{and} \quad C^{0}_{T}=\left\{X\in C^{0}(\mathbb{R},\mathbb{R}^{2}): X(t+T)=X(t)\right\},
\]
and define the functional $\displaystyle{\mathcal{H}:C^{1}_{T}\times \mathbb{R}^{2}\to C^{0}_{T}}$ given by 
\[
\mathcal{H}(X,G):=\dot{X}-F(\cdot,X,X_{d},G), \quad F(t,X,X_{d},G)=\begin{pmatrix}
\dot{x}\\
1-\frac{e \mathcal{V}^2(t,s_{1}(x,x_{d},g_{1}),s_{2}(\dot{x},\dot{x}_{d},g_{2}))}{x^2}-x-c\dot{x}
\end{pmatrix}
\]
The function $\mathcal{H}$ is well defined and continuous. Moreover, $\displaystyle{\partial_{X}}\mathcal{H}(X,G)\in L(C^{1}_{T},C^{0}_{T})$ and $\displaystyle{\partial_{G}\mathcal{H}(X,C^{0}_{T})\in L(\mathbb{R}^{2}, C^{0}_{T})}$ are given by 
\[
\partial_{X}\mathcal{H}(X,G)(Y):=\dot{Y}-\partial_{X}F(\cdot,X,X_{d},G)Y \quad \text{and} \quad  \partial_{G}\mathcal{H}(X,G)(\mathcal{G})=-\partial_{G}F(\cdot,X,X_{d},G)\mathcal{G},
\]
are continuous, therefore $\mathcal{H}$ is continuously differentiable with $D\mathcal{H}:C^{1}_{T}\times \mathbb{R}^{2}\to L(C^{1}_{T}\times \mathbb{R}^{2},C^{0}_{T})$ with
\[
D\mathcal{H}(X,G)(Y,\mathcal{G})=\dot{Y}-\partial_{X}F(\cdot,X,X_{d},G)Y-\partial_{G}F(\cdot,X,X_{d},G)\mathcal{G}.
\]
Notice that the equation $\mathcal{H}(X,0)=0$ has the two nontrivial $T$-periodic solutions $\Psi_{1,2}=\Psi_{1,2}(t)$. Let $\displaystyle{\Psi^{\ast}\in \left\{\Psi_{1},\Psi_{2}\right\}}$ and consider for example the function $\Psi^{\ast}(t)=\Psi_{1}(t)$, then from Theorem \ref{Periodic solutions existence} it follows that the linear operator $\displaystyle{\partial_{X}\mathcal{H}(\Psi^{\ast},0):C^{1}_{T}\to C^{0}_{T}}$ with
\[
\begin{split}
\partial_{X}\mathcal{H}(\Psi^{\ast},0)(X(t))&=\dot{X}(t)-\partial_{X}F(t,\Psi^{\ast},\Psi^{\ast}_{0},0)X(t),\\
&=\dot{X}(t)-\begin{pmatrix}
0& 1 \\
\frac{2e V_{\delta}^{2}(t)}{\psi^{3}_{1}(t)}-1& -c
\end{pmatrix}X(t),
\end{split}
\]
is one-to-one. (See Remark \ref{non nontrivial periodic sol}) Then, by the Fredholm Alternative, for each function $Y\in C^{0}_{T}$ there exists a unique solution  $W_{Y}\in C^{1}_{T}$ of the equation
\[
\partial_{X}\mathcal{H}(\Psi^{\ast},0)(W)=Y.
\]
Now, by the Open Mapping Theorem, $\displaystyle{\big[\partial_{X}\mathcal{H}(\Psi^{\ast},0)\big]^{-1}}:C^{0}_{T}\to C^{1}_{T}$, $Y\to W$  is continuous implying that $\displaystyle{\partial_{X}\mathcal{H}(\Psi^{\ast},0)}$ is an isomorphism. In consequence, by the Implicit Function Theorem, there exists a neighborhood $\mathcal{V}\subset \R$ of $(0,0)$ and a unique $C^{1}$ function $\Gamma: \mathcal{V}\to C^{1}_{T}$ such that
\begin{equation}\label{funcion Phi}
\Gamma(0)=\Psi^{\ast}, \quad \text{and} \quad  \mathcal{H}(\Gamma(G),G)=0, 
\end{equation}
for all $G \in \mathcal{V}.$ The same conclusion is obtained if $\Psi^{\ast}=\Psi_{2}$. Finally, the proof  follows if we define $X_{G}=\Gamma(G).$
\end{proof}

\begin{remark}
By the Theorem of Continuous Dependence of Parameters, the $T$-periodic solution $X^{G}=X^{G}(t)$ with $X^{0}(t)=\Psi_{2}(t)$ $(resp. X^{0}(t)=\Psi_{1}(t)$) given by Theorem \ref{main-theo 1} is stable (unstable).  In addition, although this result is true for any fixed $d\in \mathbb{R}$, actually we are assuming $0\leq d<T$ (see Remark \ref{countable}). 
\end{remark}

In contrast with the existence of periodic solutions for \eqref{system nathanson_norm feedback} given by Theorem \ref{main-theo 1}, which is valid for any fixed value of the delay $d$ and with small gains $g_1$, $g_2$, in what follows we are interested in continuation of periodic solutions of \eqref{system nathanson_norm feedback} using the delay as a continuation parameter. Certainly, some appropriate conditions over the gains are expected.  To this end, we need the following preliminary results.

\begin{lemma}\label{lema cota de la derivada}
Let $\Psi(t)=(\psi(t),\dot{\psi}(t))^{tr}$ be a $T$-periodic solution of the differential system  \eqref{system nathanson_norm non-feedback} in the case $h_{D}(x,\dot{x})=c\dot{x}$. Assume that $\displaystyle{0<\varsigma\leq \psi(t)\leq \varrho}$ for all $t\in \R$. Then, $\displaystyle{|\dot{\psi}(t)|\leq \Lambda(\varsigma,\varrho)}$ for all $t\in \mathbb{R},$ where  
\begin{equation}\label{cota derivada}
\Lambda(\varsigma,\varrho):= \gamma(\varrho-\varsigma)+\max\left\{\Big|1-\varsigma-\frac{e V^{2}_{\delta,\min}}{\varrho^{2}}\Big|,\Big|1-\varrho-\frac{e V^ {2}_{\delta,\max}}{\varsigma^{2}}\Big|\right\} T.
\end{equation}
\end{lemma}
\begin{proof}
Notice that the function $\psi=\psi(t)$ satisfies the differential equation
\[
\ddot{\psi}=-c\dot{\psi}+1-\psi-\frac{e V_{\delta}^{2}(t)}{\psi^{2}},
\]
Since $\psi(t)$ is $T$-periodic, there exists $t_{0}\in [0,T]$ such that $\dot{\psi}(t_{0})=0$. Then integrating the equation between $t_{0}$ and $t$ we have
\[
\dot{\psi}(t)=-c\big(\psi(t)-\psi(t_{0})\big)+\int_{t_{0}}^{t}\big(1-\psi(s)-\frac{eV_{\delta}^{2}(s)}{\psi^{2}(s)}\Big)ds.
\]
Now by hypothesis, $0<\varsigma\leq \psi(t)\leq \varrho $ for all $t\in \R$, then
\[
1-\varsigma-\frac{eV_{\delta,\min}^{2}}{\varrho^{2}}>1-\psi(t)-\frac{eV_{\delta}^{2}(t)}{\psi^{2}(t)}>1-\varrho-\frac{eV_{\delta,\max}^{2}}{\varsigma^{2}}, 
\]
for all $t\in \mathbb{R}.$ From here we deduce that
\[
|\dot{\psi}(t)|\leq c(\varrho-\varsigma)+\max\left\{\Big|1-\varsigma-\frac{eV_{\delta,\min}^{2}}{\varrho^{2}}\Big|,\Big|1-\varrho-\frac{eV_{\delta,\max}^{2}}{\varsigma^{2}}\Big|\right\}T,
\]
for all $t\in \mathbb{R.}$
\end{proof}

\begin{lemma}\label{lema hill} Let $c>0$ and $a,b\in C^{1}(\mathbb{R}/T \Z)$. Assume that the following assumptions hold
\begin{enumerate}
\item[1.] $\displaystyle{4a(t)\geq (c+b(t))^{2}+2(\dot{b}(t))^{2}}$.
\item[2.] $\displaystyle{\|4a(t)-(c+b(t))^{2}-2(\dot{b}(t))^{2} \|_{L^{\infty}}\leq \Big(\frac{\pi}{T}\Big)^{2}.}$ 
\item[3.] $\displaystyle{c+\overline{b}\neq 0,}$ with $\displaystyle{\overline{b}=\frac{1}{T}\int_{0}^{T}b(t)dt}.$
\end{enumerate}
Then, the Hill equation
\begin{equation}\label{Complete Hill equation}
\ddot{y}+(c+b(t))\dot{y}+a(t)y=0,
\end{equation}
does not admit nontrivial $T$-periodic solutions.
\end{lemma}

\begin{proof}Let $y(t)$ any nontrivial solution of \eqref{Complete Hill equation}. A direct computation shows that the function $z=z(t)$ given by
\[
z(t)=y(t)w(t), \quad \text{with} \quad w(t)=\exp\left\{\frac{1}{2}\int_{0}^{t}(c+b(s))ds\right\},
\]
satisfies the Hill equation
\begin{equation}\label{Canonical Hill equation}
\ddot{z}+Q(t)z=0, \quad \text{with} \quad Q(t)=a(t)-\frac{1}{4}(c+b(t))^{2}-\frac{1}{2}(\dot{b}(t))^{2}.
\end{equation}
By assumptions 1. and 2. the function $Q(t)$ has the properties
\[
0<Q(t), \quad \text{and} \quad \|Q\|_{L^{\infty}}\leq \Big(\frac{\pi}{T}\Big)^{2}.
\]
Then, by the Lyapunov-Borg criteria (see \cite{cabada,Zhang-Li}) the solutions of \eqref{Canonical Hill equation} are bounded. 
On the other hand, for all $t\in \mathbb{R}$ we have
\[
w(t+T)=\upsilon w(t), \quad \text{with} \quad \upsilon= \exp\left\{\frac{T(c+\overline{b})}{2}\right\}.
\]
From here and assumption 2. we deduce that $w(t)$ is not bounded. Finally, if $y(t)$ is a bounded solution of \eqref{Complete Hill equation} then $w(t)$ is bounded, which is a contradiction.
In particular, $y(t)$ cannot be periodic unless it is trivial.
\end{proof} 

Now we are able to prove the local continuation of periodic solutions using the delay as a continuation parameter. Recall that, from Theorem \ref{Periodic solutions existence}, there exist exactly two periodic solutions $\Psi_{i}(t)=(\psi_{i}(t),\dot{\psi}_{i}(t))^{tr}$, $i=1,z$ of \eqref{system nathanson_norm non-feedback}. Moreover, in the case $h_{D}(x,\dot{x})=c\dot{x}$ we have
\[
\eta_{1}<\psi_{1}(t)<\xi_{1}, \quad \text{and} \quad \xi_{2}<\psi_{2}(t)<\eta_{2},
\]
for all $t\in \R$, where $\xi_{i}$ and $\eta_{i}$, $i=1,2$ are the solutions in $]0,1[$ of the equation
\begin{equation}\label{ecuaciones de los super y sub soluciones}
(1-x)x^{2}=eV^{2}_{\delta,\max}, \quad (1-x)x^{2}=eV^{2}_{\delta,\min},
\end{equation}
respectively. Let us consider a continuation result from the starting periodic function $\Psi_{2}(t)$. An analogous result can be obtained starting from $\Psi_{1}(t)$. 

\begin{theorem}[Continuation of periodic solutions over the delay]\label{main-theo 2}
Let $\Psi_{2}(t)=(\psi_{2}(t),\dot{\psi}_{2}(t))^{tr}$ be the $T$-periodic solution of \eqref{system nathanson_norm non-feedback} given by Theorem \ref{Periodic solutions existence} in the case $h_{D}(x,\dot{x})=c\dot{x},$ $c>0.$ For $g_{2}>0$, let us define 
\[
b^{\ast}=\frac{2g_{2}(1-\xi_{2})}{V_{\delta,\max}}, \quad \text{and} \quad \dot{b}^{\ast}=\frac{2eg_{2}}{\xi^{3}_{2}}\Big(\xi_{2}\|\dot{V}_{\delta}\|_{\infty}+2V_{\delta,\max}\Lambda(\xi_{2},\eta_{2})\Big), 
\]
with $\displaystyle{\Lambda(\xi_{2},\eta_{2}})$ given as in Lemma \ref{lema cota de la derivada}. Then, there exists a neighborhood $\Omega \subset \mathbb{R}$ of $d=0$ and a unique function $X(t,d)$, that is a $T$-periodic solution of \eqref{system nathanson_norm feedback} for all $d\in \Omega$, i.e., 
{\renewcommand{\labelitemi}{$\triangleright$}
\begin{itemize}
    \item $\dot{X}(t,d)=F(t,X(t,d),X_{d}(t,d)), \quad X(t+T,d)=X(t,d),$ for all $d \in \Omega.$
    \item $X(t,0)=\Psi_{2}(t)$ for all $t\in \mathbb{R}$, 
\end{itemize}}
if $g_{2}>0$ and one of the following conditions hold
\begin{itemize}
\item[a)] If $g_{1}>0$:
\begin{equation}\label{condition for hill}
\frac{1}{4}(c+b^{\ast})^{2}+\frac{1}{2}(\dot{b}^{\ast})^{2}\leq \frac{2g_{1}(1-\eta_{2})}{V_{\delta,\min}}+\frac{3\xi_{2}-2}{\xi_{2}}, \quad \text{and} \quad \frac{2g_{1}(1-\xi_{2})}{V_{\delta,\max}}+\frac{3\eta_{2}-2}{\eta_{2}}<\Big(\frac{\pi}{T}\Big)^{2}+\frac{c^{2}}{4}.
\end{equation}
\item[b)] If $g_{1}<0$,:
\[
\frac{1}{4}(c+b^{\ast})^{2}+\frac{1}{2}(\dot{b}^{\ast})^{2}\leq \frac{2g_{1}(1-\xi_{2})}{V_{\delta,\max}}+\frac{3\xi_{2}-2}{\xi_{2}}, \quad \text{and} \quad \frac{2g_{1}(1-\eta_{2})}{V_{\delta,\min}}+\frac{3\eta_{2}-2}{\eta_{2}}<\Big(\frac{\pi}{T}\Big)^{2}+\frac{c^{2}}{4}.
\]
\end{itemize}
Moreover, if $g_{2}<0$ and the damping coefficient $c$ satisfies
\begin{equation}\label{condiciones sobre c}
c+\frac{2g_{2}(1-\eta_{2})}{V_{\delta,\min}}<0,  \quad \text{or} \quad c+\frac{2g_{2}(1-\xi_{2})}{V_{\delta,\max}}>0,
\end{equation}
and one of the following conditions hold
\begin{itemize}
\item[c)] For $g_{1}>0$:
\begin{equation*}
\begin{split}
0<g_{2}+\frac{c V_{\delta,\max}}{1-\xi_{2}}, & \qquad  \frac{1}{4}(c+b^{\ast})^{2}+\frac{1}{2}(\dot{b}^{\ast})^{2}<\frac{2g_{1}(1-\eta_{2})}{V_{\delta,\min}}+\frac{3\xi_{2}-2}{\xi_{2}} \quad \text{and}\\
&\frac{2g_{1}(1-\xi_{2})}{V_{\delta,\max}}+\frac{3\eta_{2}-2}{\eta_{2}}<\Big(\frac{\pi}{T}\Big)^{2}+\frac{c^{2}}{4}.
\end{split}
\end{equation*}
\item[d)] For $g_{1}<0$:
\begin{equation*}
\begin{split}
0<g_{2}+\frac{c V_{\delta,\max}}{1-\xi_{2}}, & \qquad  \frac{1}{4}(c+b^{\ast})^{2}+\frac{1}{2}(\dot{b}^{\ast})^{2}<\frac{2g_{1}(1-\xi_{2})}{V_{\delta,\max}}+\frac{3\xi_{2}-2}{\xi_{2}} \quad \text{and}\\
&\frac{2g_{1}(1-\eta_{2})}{V_{\delta,\min}}+\frac{3\eta_{2}-2}{\eta_{2}}<\Big(\frac{\pi}{T}\Big)^{2}+\frac{c^{2}}{4}.
\end{split}
\end{equation*}
\end{itemize}
we obtained the same conclusion.
\end{theorem}

\begin{proof}
The proof follows the same ideas given in the proof of Theorem 
\ref{main-theo 1}, one of the main difference is that now we consider the functional $\mathcal{F}:C^{1}_{T}\times \mathbb{R} \to C^{0}_{T}$ while in Theorem \ref{main-theo 1} we considered the Banach spaces $C^{1}_{T}\times \mathbb{R}^2$ and $C^{0}_{T}$. 
\[
\mathcal{F}(X,d):=\dot{X}-F(\cdot,X,X_{d},d), \quad F(t,X,X_{d},d)=\begin{pmatrix}
\dot{x}\\
1-\frac{e \mathcal{V}^2(t,s_{1}(x,x_{d},g_{1}),s_{2}(\dot{x},\dot{x}_{d},g_{2}))}{x^2}-x-c\dot{x}
\end{pmatrix}
\]

As before, the equation $\mathcal{F}(X,0)=0$ has two nontrivial $T$-periodic solutions given by  $\Psi_{1,2}=\Psi_{1,2}(t)$. In particular, if $\Psi^{\ast}(t)=\Psi_{2}(t)$, the linear operator $\displaystyle{\partial_{X}\mathcal{F}(\Psi^{\ast},0):C^{1}_{T}\to C^{0}_{T}}$ given by
\[
\begin{split}
\partial_{X}\mathcal{F}(\Psi^{\ast},0)(X(t))&=\dot{X}(t)-\partial_{X}F(t,\Psi^{\ast},\Psi^{\ast}_{0},0)X(t),\\
&=\dot{X}(t)-\begin{pmatrix}
0& 1 \\
\frac{2e V_{\delta}^{2}(t)}{\psi^{3}_{2}(t)}-\frac{2eg_{1}V_{\delta}(t)}{\psi^{2}_{2}(t)}-1& -c-\frac{2eg_{2}V_{\delta}(t)}{\psi^{2}_{2}(t)}
\end{pmatrix}X(t).
\end{split}
\]
The objective is to prove that $\displaystyle{\partial_{X}\mathcal{F}(\Psi^{\ast},0)}$ is inyective. To this end, let us consider the equation $\displaystyle{\partial_{X}\mathcal{F}(\Psi^{\ast},0)(Y(t))=0,}$ which is equivalent to the Hill equation  \begin{equation}\label{ecuación de hill-nathanson}
    \ddot{y}+(c+b(t))\dot{y}+a(t)y=0,
\end{equation}
with
\[
a(t)=1-\frac{2e V_{\delta}^{2}(t)}{\psi^{3}_{2}(t)}+\frac{2eg_{1}V_{\delta}(t)}{\psi^{2}_{2}(t)}, \quad \text{and} \quad b(t)=\frac{2eg_{2}V_{\delta}(t)}{\psi^{2}_{2}(t)}.
\]
By Theorem \ref{Periodic solutions existence} the $T$-periodic function $\psi_{2}(t)$ satisfies
\[
\xi_{2}\leq \psi_{2}(t)\leq \eta_{2}, \quad \forall t \in \mathbb{R}.
\]
where $\xi_{2}$, $\eta_{2}$ are the solutions in $[2/3,1]$ of 
\[
(1-x)x^{2}=eV^{2}_{\delta,\max}, \quad (1-x)x^{2}=eV^{2}_{\delta,\min},
\]
respectively.  At this point we assume that $g_{1}, g_{2}$ are both positive so:

\textbf{case a)}. Computations show that
\[
0<a_{\ast}<a(t)<a^{\ast}, \quad \forall t\in \mathbb{\R}, 
\]
with
\[
a_{\ast}:=\frac{2g_{1}(1-\eta_{2})}{V_{\delta,\min}}+\frac{3\xi_{2}-2}{\xi_{2}},\quad \text{and} \quad a^{\ast}:=\frac{2g_{1}(1-\xi_{2})}{V_{\delta,\max}}+\frac{3\eta_{2}-2}{\eta_{2}}.
\]
Also, we have 
\[
0<b_{\ast}<b(t)<b^{\ast}, \quad \forall t\in \mathbb{\R}, 
\]
with
\[
b_{\ast}:=\frac{2g_{2}(1-\eta_{2})}{V_{\delta,\min}}, \quad \text{and} \quad b^{\ast}:=\frac{2g_{2}(1-\xi_{2})}{V_{\delta,\max}}.
\]
Likewise, for the function 
\[
\dot{b}(t)=2eg_{2}\Big(\frac{\dot{V}_{\delta}(t)}{\psi_{2}^{2}(t)}-2\frac{V_{\delta}(t)\dot{\psi}(t)}{\psi_{2}^{3}(t)}\Big),
\]
we have
\[
|\dot{b}(t)|\leq \dot{b}^{\ast}:= \frac{2eg_{2}}{\xi^{3}_{2}}\Big(\xi_{2}\|\dot{V}_{\delta}\|_{\infty}+2V_{\delta,\max}\Lambda(\xi_{2},\eta_{2})\Big), \quad \forall t\in \mathbb{\R},
\]
where $\displaystyle{\Lambda(\xi_{2},\eta_{2})}$ is given as in Lemma \ref{lema cota de la derivada}. Then, by the assumptions \eqref{condition for hill} it follows that
\[
Q(t):=a(t)-\frac{1}{4}(c+b(t))^{2}-\frac{1}{2}(\dot{b}(t))^{2}>a_{\ast}-\frac{1}{4}(c+b^{\ast})^{2}-\frac{1}{2}(\dot{b}^{\ast})^{2}>0, \quad \forall t\in \mathbb{R},
\]
and
\[
\|Q\|_{L^{\infty}}=\|a(t)-\frac{1}{4}(c+b(t))^{2}-\frac{1}{2}(\dot{b}(t))^{2}\|_{L^{\infty}}<\|a(t)-\frac{c^{2}}{4}\|_{L^{\infty}}<\Big(\frac{\pi}{T}\Big)^{2}.
\]
Since $b(t)>0$ then $c+\overline{b}>0$, in consequence by Lemma \ref{lema hill} the equation \eqref{ecuación de hill-nathanson} does not admit nontrivial $T$-periodic solutions, implying that $\partial_{X}\mathcal{F}(\Psi^{\ast},0)$ is injective. The proof proceeds in the same fashion as the proof of Theorem \ref{main-theo 1}, showing the existence of a small neighborhood $\Omega$ of $d=0$ and a unique function $X(\cdot,d):\R\to \R$ such that
\[
\dot{X}(t,d)=F(t,X(t,d),X_{d}(t,d)), \quad X(t+T,d)=X(t,d),\quad \forall  t\in \R, \, \forall d \in \Omega,
\]
i.e., $X(t,d)$ is a $T$-solution of \eqref{system nathanson_norm feedback}. Moreover, $X(t,0)=\Psi_{2}(t)$ for all $t\in \R.$

\textbf{case b)}. The proof proceeds analogously at the previous one.

\textbf{case c)}. If $g_{2}<0$ notice that $b(t)<0$ and moreover
\[
c+b^{\ast}<c+\overline{b}<c+b_{\ast}<c,
\]
therefore any of the inequalities in \eqref{condiciones sobre c} will imply that $c+\overline{b}\neq 0.$ Thus, the assumption 2. of Lemma \ref{lema hill} is fulfill. From this the rest of the proof is verbatim the case $a)$ under the respective and additional assumptions.

\textbf{case d)}. The proof proceeds analogously to the previous one.
\end{proof}

\section*{4. Numerical validation}

As a final contribution, we present some numerical computations in order to validate the main results of the document concerning the existence and stability of periodic solutions of \eqref{system nathanson_norm feedback} and the stability of the equilibrium $(x_{2},0)$ of (\ref{eq:delayed system}). In Table \ref{tab:parametros} we list  the values of the parameters that we have taken from \cite{Younis-2010,Younis-2013} for the delayed Nathanson's equation
\begin{equation}
\label{eq:nathanson}
m y^{\prime \prime}(\tau) + \xi y^{\prime}(\tau)+ k y(\tau) = \dfrac{\epsilon A \tilde{V}_{\delta}^{2}(\tau)}{2(l-y(\tau))^2}, \qquad \tilde{V}_{\delta}(\tau)=V_{\delta}(\tau)+\sum_{i=0}^{1}\tilde{G}_{i}\big(y^{i)}_{\tilde{d}}(\tau)-y^{i)}(\tau)\big)
\end{equation}
Introducing the following dimensionless variables 
\[
\label{eq:norm_var}
y= (1-x)l, \qquad \tau = t \mathcal{T}, \qquad \tilde{d}=d\mathcal{T}, \qquad \mathcal{T}=\sqrt{m/k},
\]
the corresponding dimensionless equation from \eqref{eq:nathanson} is the equation \eqref{eq:nathanson_norm feedback 2} with
\[
c=\dfrac{\xi}{\sqrt{m k}}, \quad e = \dfrac{\epsilon A}{2kd^3}, \quad G_{1}=\tilde{G}_{1}l,\quad G_{2}=\dfrac{\tilde{G}_{2}l}{\mathcal{T}}, \quad \text{and} \quad  \mathcal{V}_{\delta}(t) = \tilde{V}_{\delta}( t\mathcal{T}).
\]
Let us assume $V_{\delta}(t)=20+\delta\cos(\Omega t)$ with $\Omega<2.$ Then, for the autonomous case ($\delta=0$) we compute the equilibrium points  given as solutions $x_{1},x_{2}$ of the equation $x^{2}(1-x)=400e.$
\begin{table}[h]
	\begin{center}
	\begin{normalsize}
	\begin{tabular}{|c|c|c|c|}
		\hline
		$m$ [Kg] & $l$ [m]& $A$ [m$^{2}$]& $k$ [N/m]\\
		\hline
		$21 \times 10^{-5}$ & $3.8 \times 10^{-5}$ & $3.96\times 10^{-5}$ & $320$\\
	 \hline
	 $\xi$ [N s m$^{-1}$]	& $\varepsilon$ [F/m] & $G_{1}$[V m$^{-1}$] & $G_{2}$[V s m$^{-1}$]\\
	 \hline
	 $0.0014$ & $8.85\times 10^{-12}$ & 8 & 8 \\ \hline
		\end{tabular}
		\end{normalsize}
		\end{center}
		\caption{Values of the parameters for equation \eqref{eq:nathanson}.}	\label{tab:parametros}
	    \end{table}

\begin{table}[h]
	\begin{center}
	\begin{normalsize}
	\begin{tabular}{|c|c|c|c|c|c|}
		\hline
		$c$ & $e$ & $x_{1}$ & $x_{2}$ & $G_{1}$& $G_{2}$\\
		\hline
		$5.4 \times 10^{-3}$ & $9.9 \times 10^{-6}$ & $6.5 \times 10^{-2}$ & $9.95 \times 10^{-1}$ & $3\times 10^{-4}$ & $3.7\times 10^{-1}$\\
	 \hline
		\end{tabular}
		\end{normalsize}
		\end{center}
		\caption{Values of the dimensionless parameters for equation \eqref{eq:nathanson_norm feedback 2} based on the data in Table \ref{tab:parametros}.}	\label{tab:parametros2}
	    \end{table}

Using the values in Table \ref{tab:parametros2} the parameters in Theorem \ref{delay d0}, for the case of linear damping are presented in Figure \ref{tab:lindam-values}. As is indicated in Proposition \ref{Prop 1}, DC-voltage source takes values 
\begin{figure}[h!]
\centerline
{\includegraphics[scale=0.7]{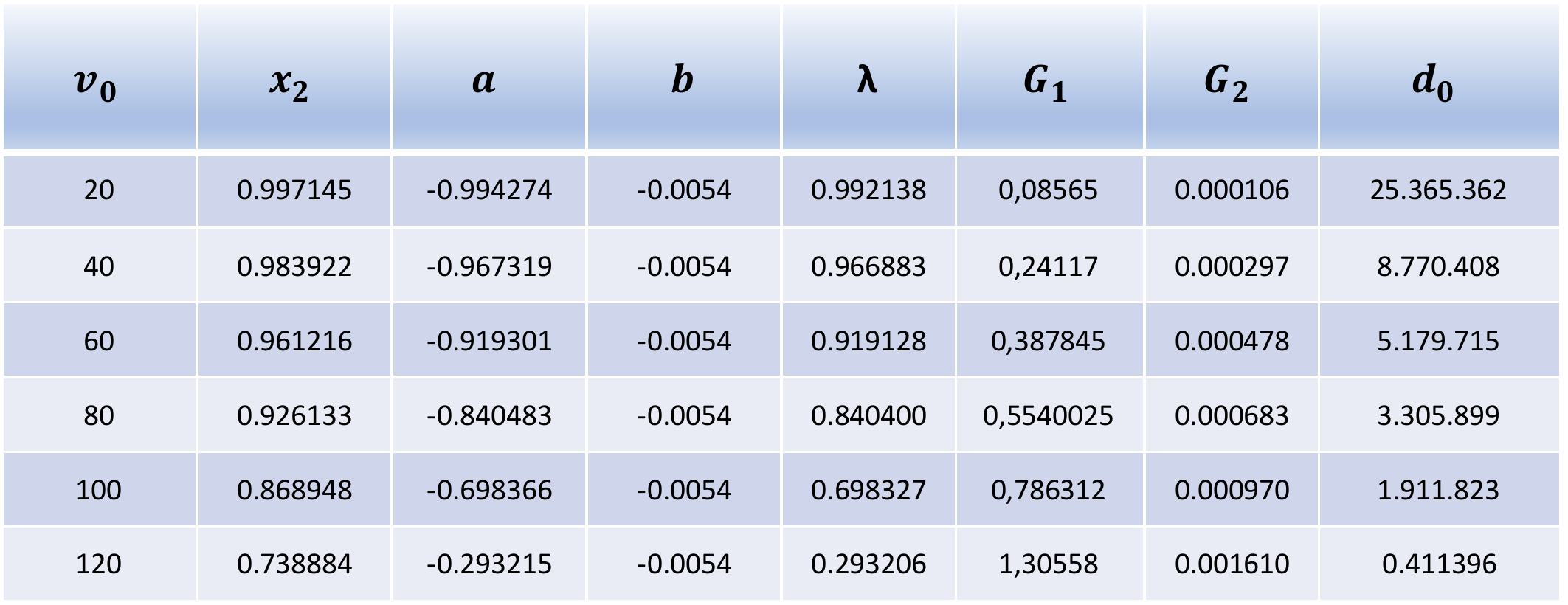}}
\caption{ 
Values for the parameters in Theorem \ref{delay d0}, for the case of a linear damping. 
}
\label{tab:lindam-values}
\end{figure}
between $0<v_0<122.33$, then taking values that vary from $10$ to $120$ and approaching the value of $x_2$ with six decimal places (the appreciable error is above the fourth decimal place). In this case, the value of the parameter $b$ is constant and we observe how the delay parameter $d_0$ decreases as the DC-voltage source increases its value.

For a squeeze film damping the parameters in Theorem \ref{delay d0} are presented in Figure \ref{tab:squeezedam-values}. At this case the value of the parameter $b$ is given by the function $-\dfrac{\lambda}{(1-x)^3}$, for the damping coefficient we have taken the value $\lambda=3\times 10^{-4}$ (see \cite{Lu_2021} for a deep discussion of the experimental obtaining of these values) and we observe how the delay parameter $d_0$ behaves like a Gaussian bell as DC-voltage source increases their value.

\begin{figure}[h!]
\centerline
{\includegraphics[scale=0.6]{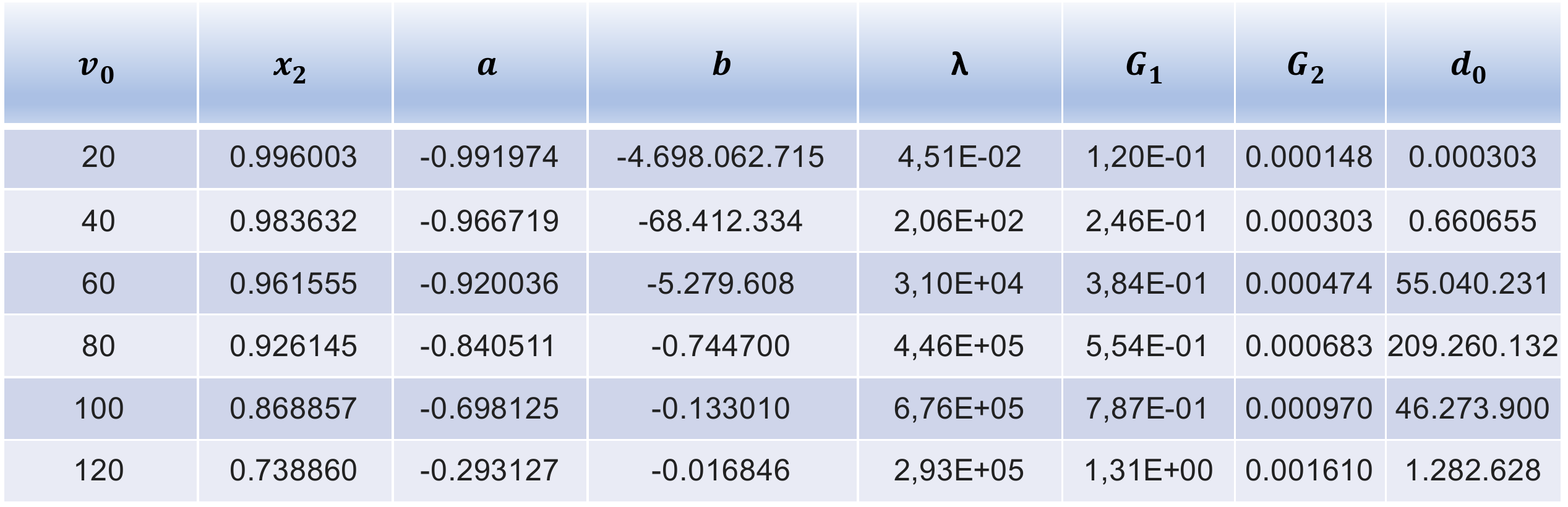}}
\caption{ 
Values for the parameters in Theorem \ref{delay d0}, for the case of a squeeze film damping force. 
}
\label{tab:squeezedam-values}
\end{figure}

As a result of the numerical analysis, we observe that for a delay values much greater than the one we predicted in our theoretical results, the asymptotic stability of the equilibrium point $(0, x_2)$ is maintained. To see this visually, we consider that  gain constants $G_1$ and $G_2$ take both, positive and negative values, as small or as large as required, this according to the different measures that are involved in the components of the MEMS.

In all the situations presented below for the autonomous case, our fixed frame will be considered equation \eqref{eq:nathanson_norm feedback 2} with linear damping for a constant $c= 0.0054$ and a fixed DC voltage value of $v_0=20$. With this we present the behavior of equation \eqref{eq:nathanson_norm feedback 2} in a neighborhood of the equilibrium point $(0, x_2)$, drawing the solutions for two initial conditions $(0, 0.2)$ and $(0, -0.2)$ which always appear one in red and one in blue. Although for practical terms it may not be fully applicable, to ensure the visual effect of the graphics, we have considered fairly large values of gain (up to 120 for $G_1$ and $G_2$) and delay (up to $d=300$ ) constants.

First, let us consider a case with fixed delay only on speed ($G_1=0$) and a positive gain constant $G_2$. The main observation is that as is observed in Figure \ref{caso:G1=0}, comparing the two graphs, a value $G_2$ is reached ($G_2=123$) for which the system changes from having an asymptotically stable equilibrium point, to one for which only stability can be assured. In Figure \ref{caso:G1=0-1} we can observe this situation for long term behavior of the initial conditions. The same result is observable if the delay is fixed but only in the position.

\begin{figure}
\centerline
{\includegraphics[scale=0.52]{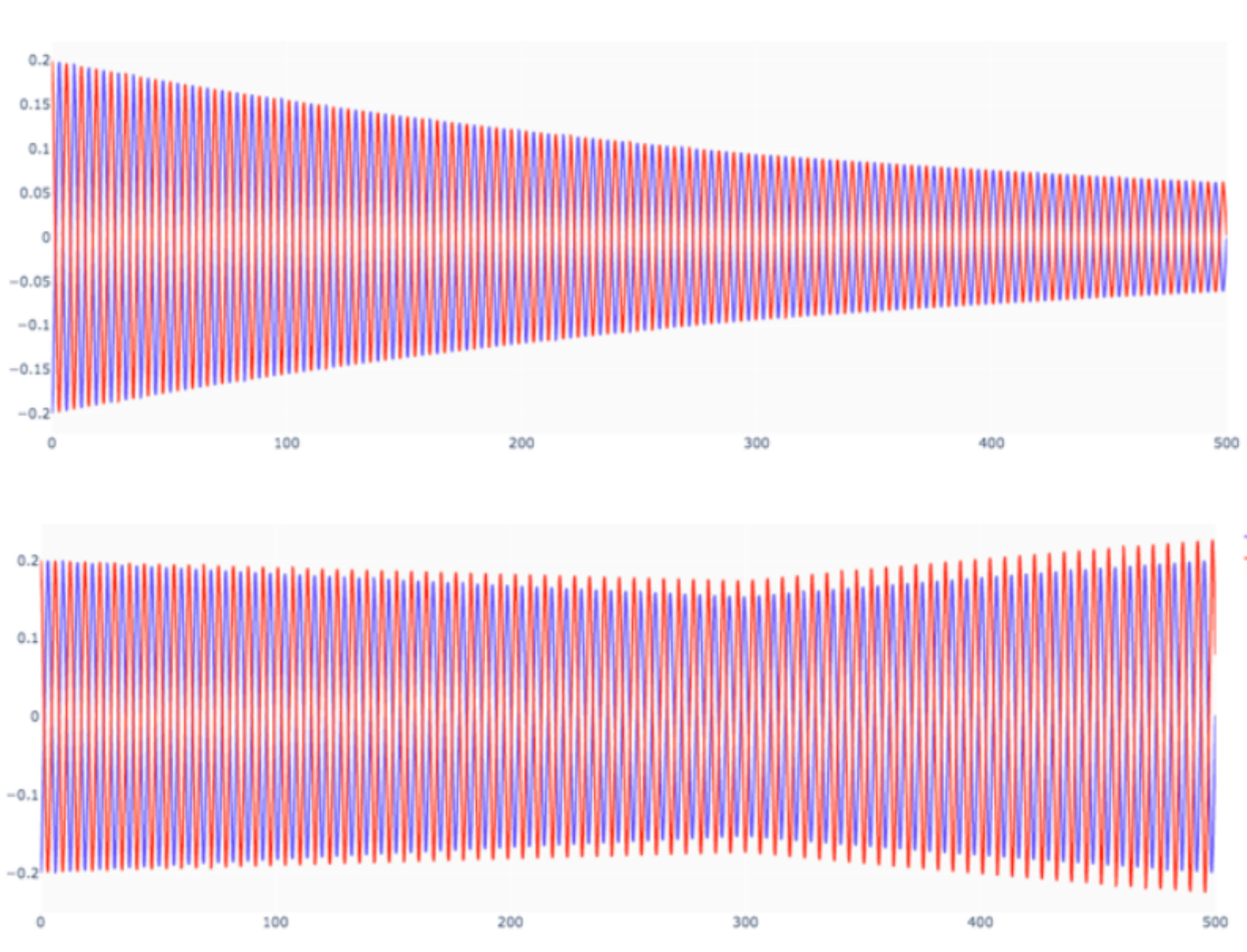}}
\caption{ 
With gain $G_1=0$ and fixed delay on speed, there is a positive gain value of $G_2$ for which the system changes from having an asymptotically stable equilibrium point to one for which only stability can be assured.
}
\label{caso:G1=0}
\end{figure}

\begin{figure}[h!]
\centerline
{\includegraphics[scale=0.45]{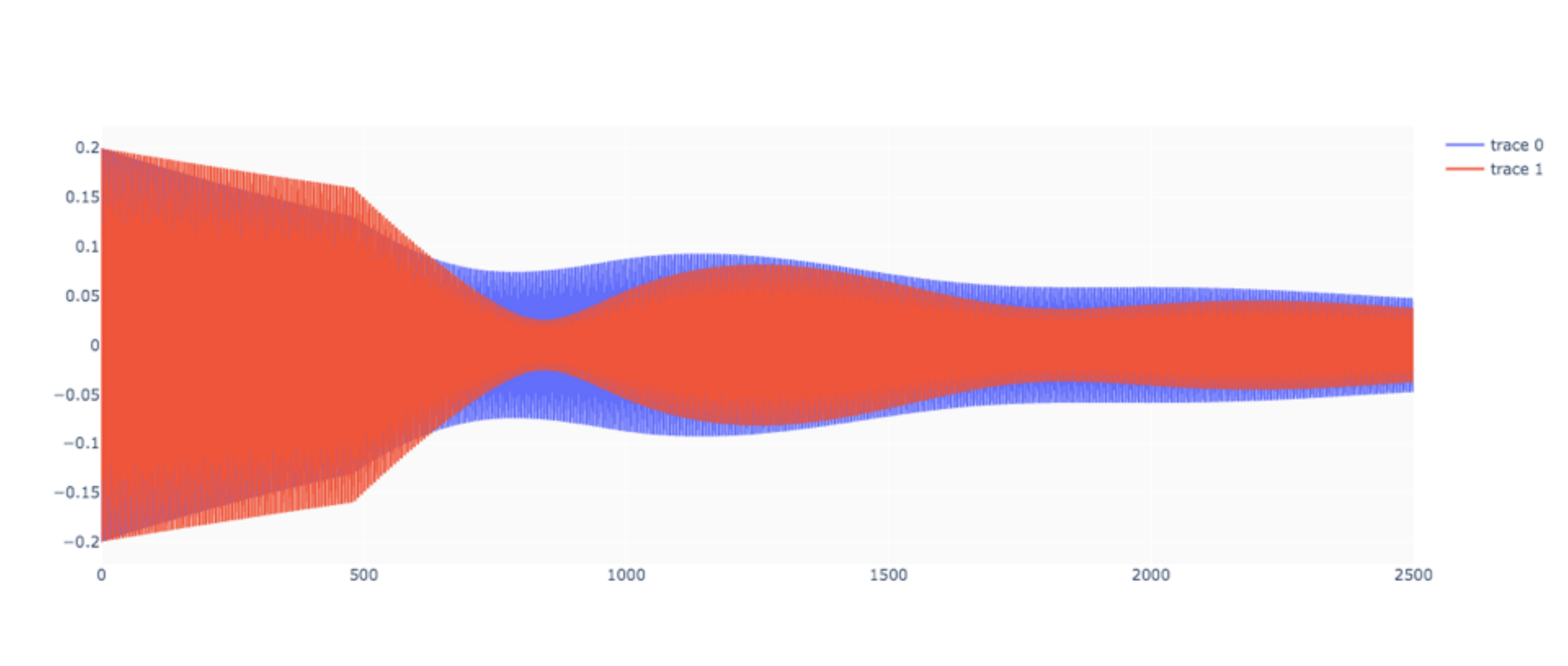}}
\caption{  
For fixed positive gain $G_1$, varying the value of the delay that is considered only in the position. The system changes from having an asymptotically stable equilibrium point to one for which only stability can be assured.
}
\label{caso:G1=0-1}
\end{figure}

We now change the situation to keep fixed positive gain constant $G_1$, varying the value of the delay that is considered only in the position, that is with $G_2=0$. As in the previous case, the main observation is that a value $d$ is reached (as observed in Figure \ref{caso:G1=0-1}) for which the system changes from having an asymptotically stable equilibrium point to one for which only stability can be assured. The decay of the system is consistent but with observable effects in a very long time.

\begin{figure}[h!]
\centerline
{\includegraphics[scale=0.45]{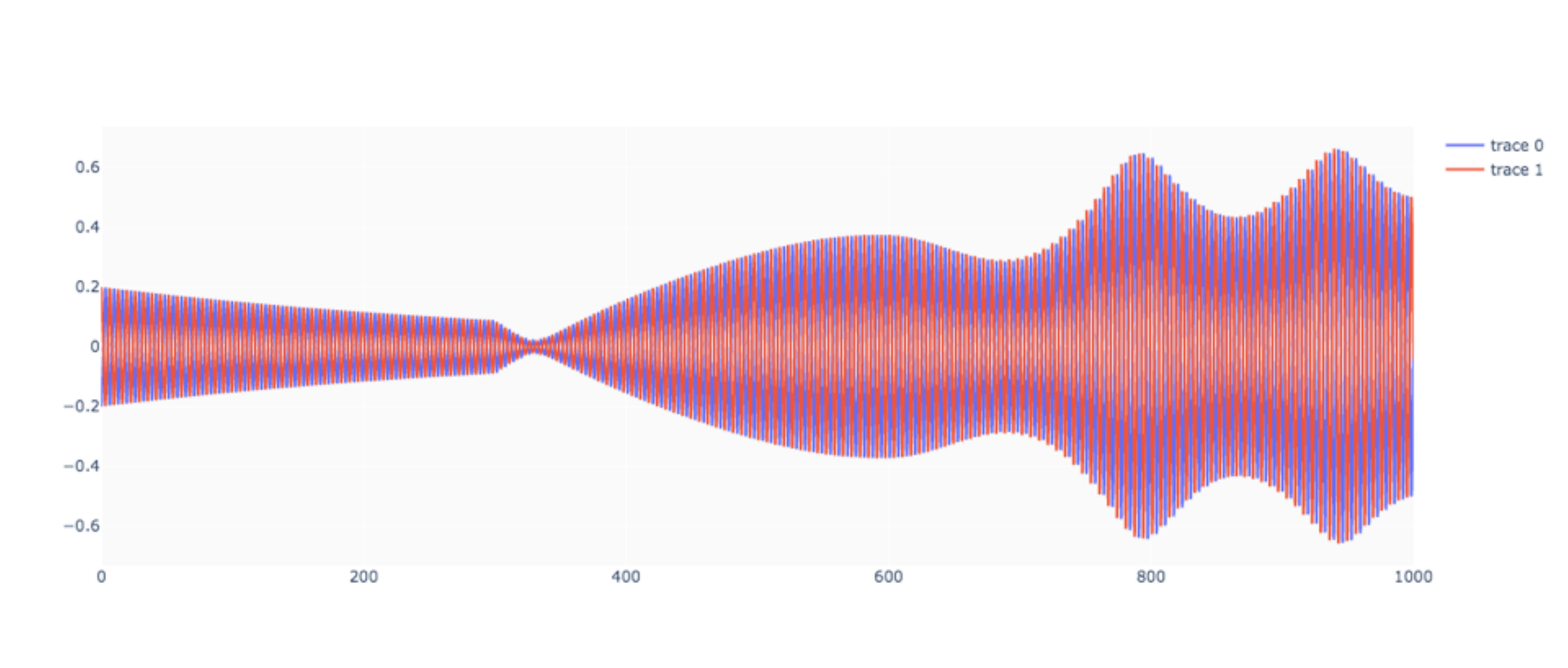}}
\caption{  
With delay fixed in the position only ($d=300$) and varying the value of the gain $G_2$ it is reached a value ($G_2=70$) in which the system, after a period of decay ($0<t<360$) gains much more energy than it previously had, reaching a large oscillation amplitude. In this case, the equilibrium point is unstable. 
}
\label{caso:G2=01}
\end{figure}

With a fixed delay in the position only and varying the value of $G_2$, as can be seen in Figure \ref{caso:G2=01}, it reaches a value of $G_2$ in which the system, after a period of decay, gains much more energy than it brought and manages, to reach a large oscillation amplitude, in this case the equilibrium point is unstable.

Under a combination of the previous cases the described behavior is persistent. For positive constants $G_1 \neq 0$, $G_2 \neq 0$ with one of these fixed and a fixed small delay value, there is a close value of the other gain constant for which the behavior of the equilibrium point changes from asymptotically stable to not. This is verified for example taking $G_1=5$ and $d=10$, where for the value $G_2=8.1$ the bifurcations hold.

The case of negative gain is even more interesting and is a source of chaotic behavior. In the graph at the top in Figure \ref{gmenos}, we observe for the two close initial conditions we are considered with fixed delay only on speed, that there is a set of negative values of $G_2$ for which both solutions evolve in such a way that decay in asymptotic behavior almost until get the minimum value of zero, this occurs for $0<t<300$, the point at which start to increase their energy again substantially to go again in decay to zero and repeating this cycle once and once, although without periodicity in time. Even though the behavior of the two solutions is qualitatively the same, the change in energy and amplitude of the oscillations is notoriously different.

\begin{figure}[H]
\centerline
{\includegraphics[scale=0.4]{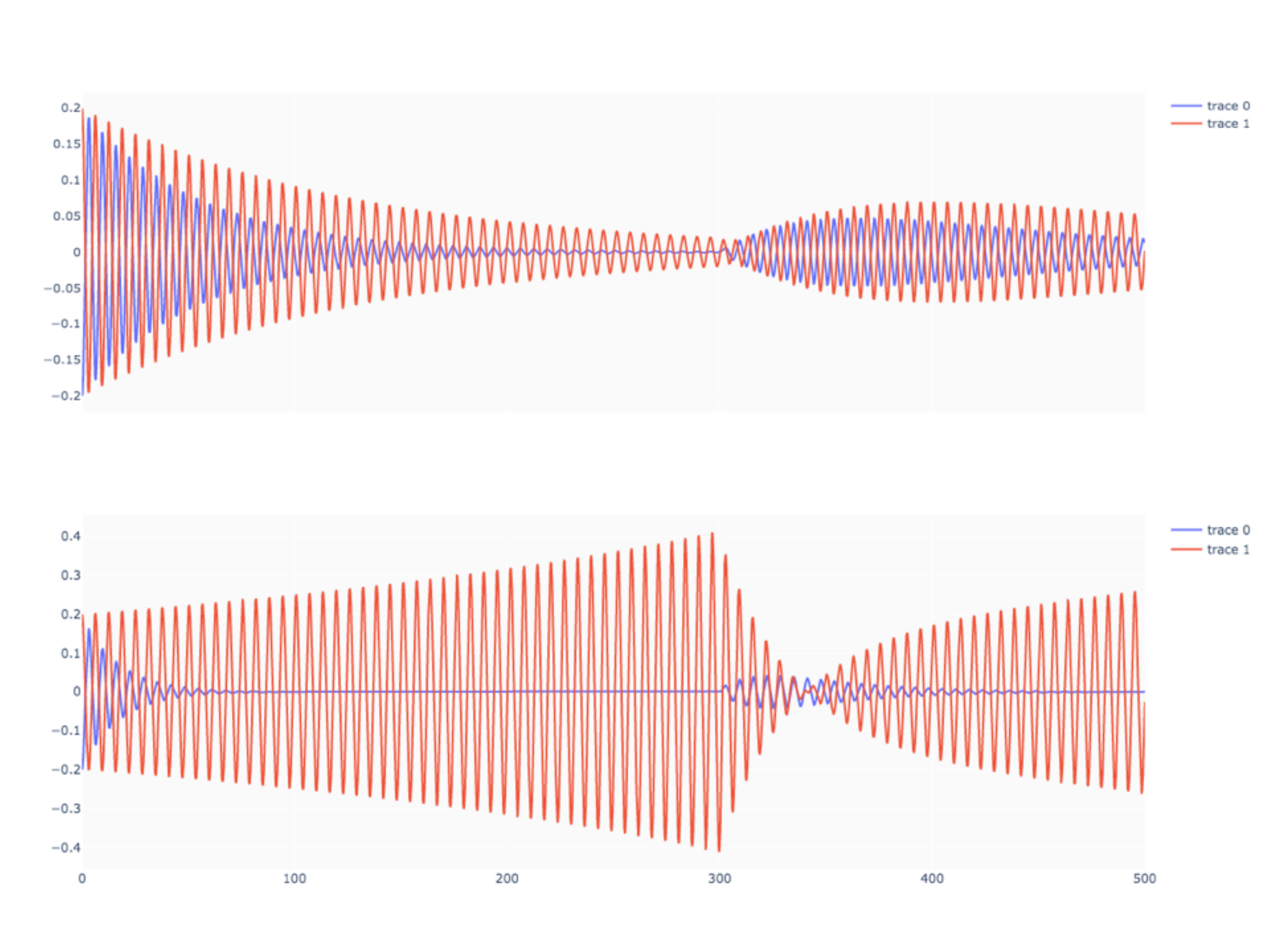}}
\caption{  
Behavior of two solutions of equation \eqref{eq:nathanson_norm feedback 2} for gain 
$G_1=0$ and fixed delay $d=300$ on speed. In the graph at the top it is observed the case $G_2=-100$ and in the graph at the bottom it is observed the case $G_2=-123$.
}
\label{gmenos}
\end{figure}

If the value of $G_2$ continues to increase, the solution corresponding to the red curve increases the value of its amplitude as the value of the gain increases (this in absolute value since the gain is being taken negative, roughly up to -118) in such a way that the asymptotic decay that was observed previously it is lost until reaching a time in which almost a harmonic oscillator regime occurs, up to a point in which it abruptly begins to lose energy during a short period of time to start to gain energy again, and this behavior of gain and loss energy alternately, is repeated cyclically. Finally, as can be seen in the graph at the bottom in Figure \ref{gmenos}, a value of $G_2$ is reached, such that contrary to the previous behavior described up to now,
the solution corresponding to the red curve starts from the initial condition increasing constantly its energy for a long period of time ($0<t<300$) in which it maintains an unstable behavior, until a point in which an abrupt loss of energy begins during a short period of time, to start to gain energy again, and this behavior of gain and loss energy alternately, is repeated cyclically. 

On the contrary, in the same interval of time in which we have described the behavior of the solution corresponding to the red curve, the solution corresponding to the blue curve loses energy very quickly giving rise to an asymptotic stability regime that is sustained for a considerable period of time ($0<t<300$), just to experiment an abrupt energy gain, which it takes the system out of stability.  This behavior of gain and loss energy alternately is repeated cyclically, which in the long term allows us to speak of the stability of the system. 

\begin{figure}[h]
\centerline
{\includegraphics[scale=0.58]{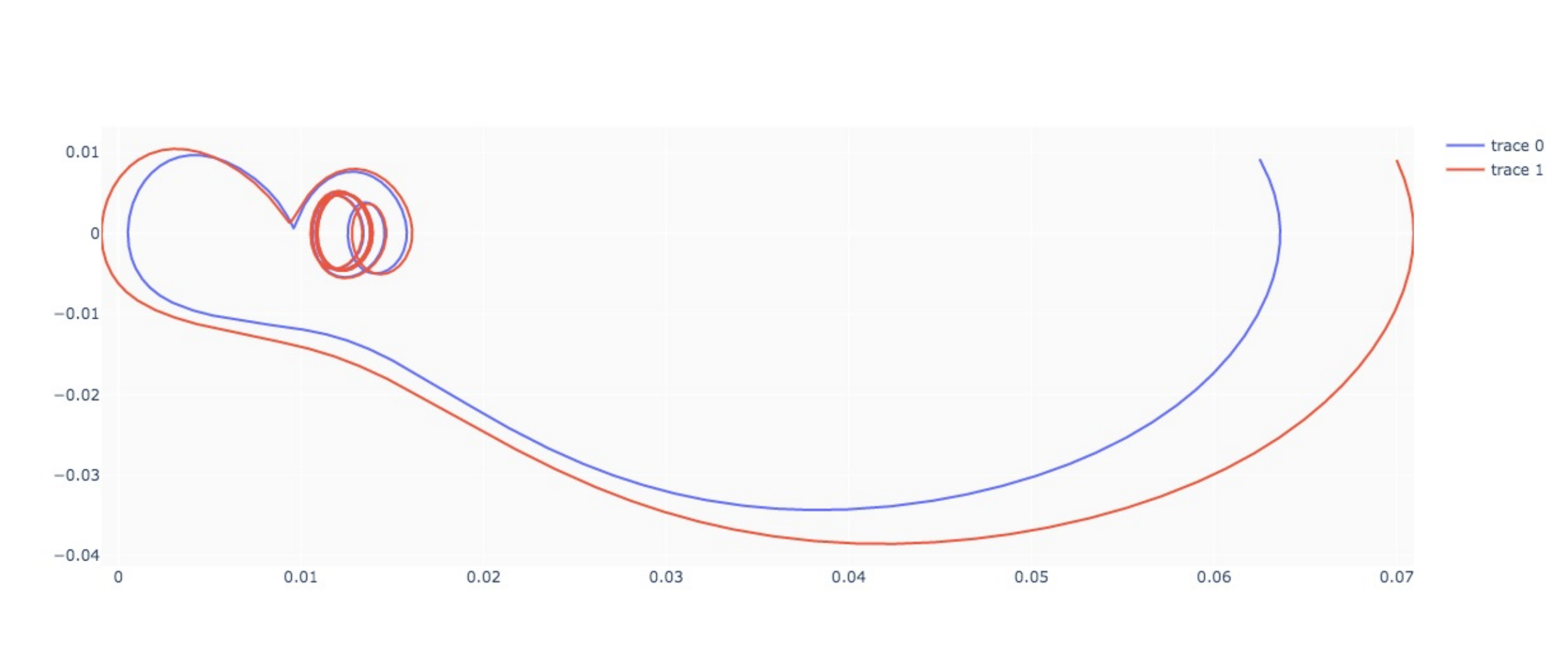}}
\caption{  
Phase portrait for two solutions of equation (\ref{system nathanson_norm non-feedback}) without delay. The asymptotically stable periodic trajectory established in Theorem \ref{main-theo 1} holds.
}
\label{periodica-d=0}
\end{figure}

 To finish with the numerical validation of the theoretical results that we have presented, Theorems \ref{main-theo 1} and \ref{main-theo 2} provide a local continuation of $T$-periodic solutions for the Nathanson model using the gain and the delay as continuation parameters. Under this setting and for the general nonautonomous case, \textit{i.e.}, taking $\delta \neq 0$, we can recreate the following scenario taking a fixed value of $\delta= 0.1579$ and two initial conditions $(0.0625, 0.0092)$ and
$(0.07, 0.0091)$: In Figure \ref{periodica-d=0} is plotted the phase portrait associated to equation (\ref{system nathanson_norm non-feedback}) without delay. 
\begin{figure}[h]
\centerline
{\includegraphics[scale=0.58]{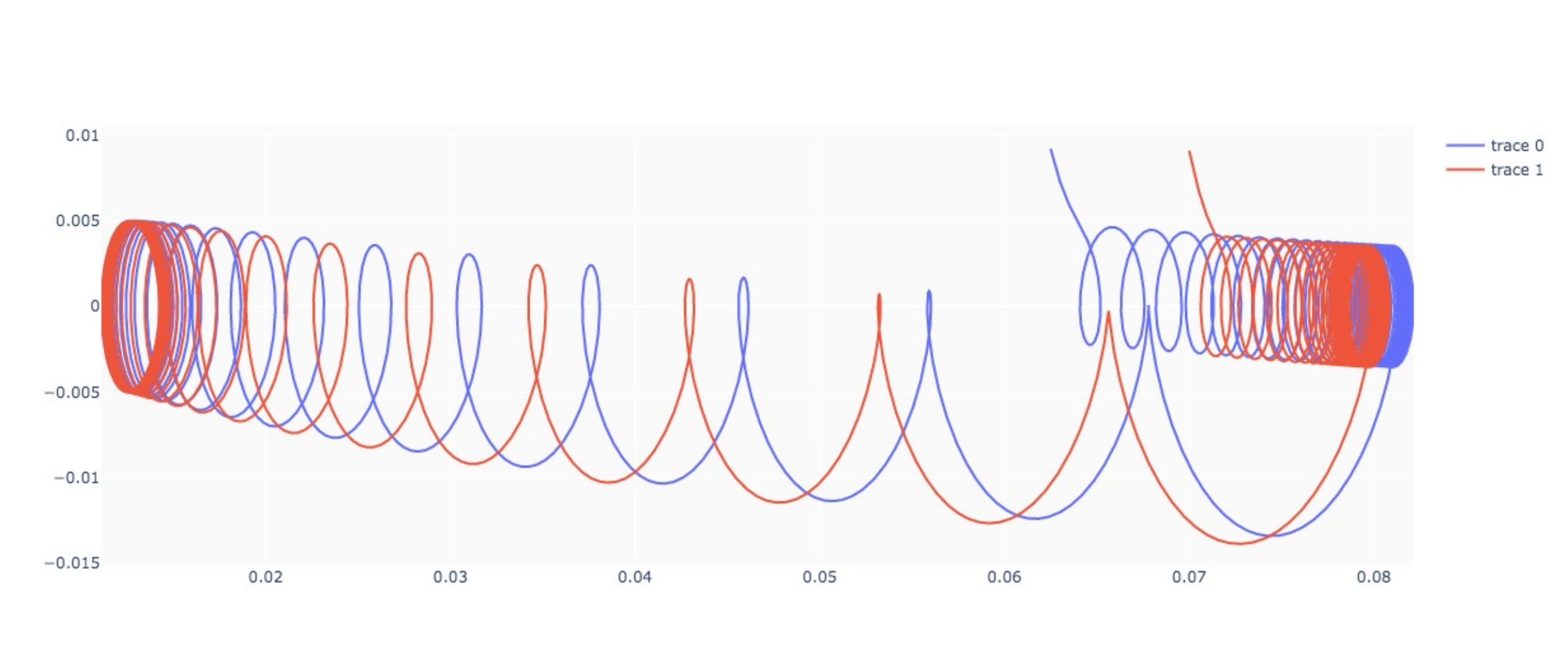}}
\caption{  
Phase portrait for two solutions of equation (\ref{system nathanson_norm non-feedback}) with delay $d=80$ and positive gain at the velocity $G_2=40$. The asymptotically stable periodic trajectory persists under the influence of these constants.
}
\label{periodica-d-Gpositivo}
\end{figure}
In this case, the solutions associated to both initial conditions approach indefinitely the asymptotically stable periodic trajectory established in Theorem \ref{main-theo 1}.

In Figure \ref{periodica-d-Gpositivo} is plotted the phase portrait associated to the equation (\ref{system nathanson_norm non-feedback}) with delay and positive gain at the velocity. In this case, the solutions associated with both initial conditions approach a region of the plane asymptotically for a considerable time to abruptly abandon it and approach the periodic trajectory established in Theorem \ref{main-theo 2}.

\begin{figure}[h]
\centerline
{\includegraphics[scale=0.5]{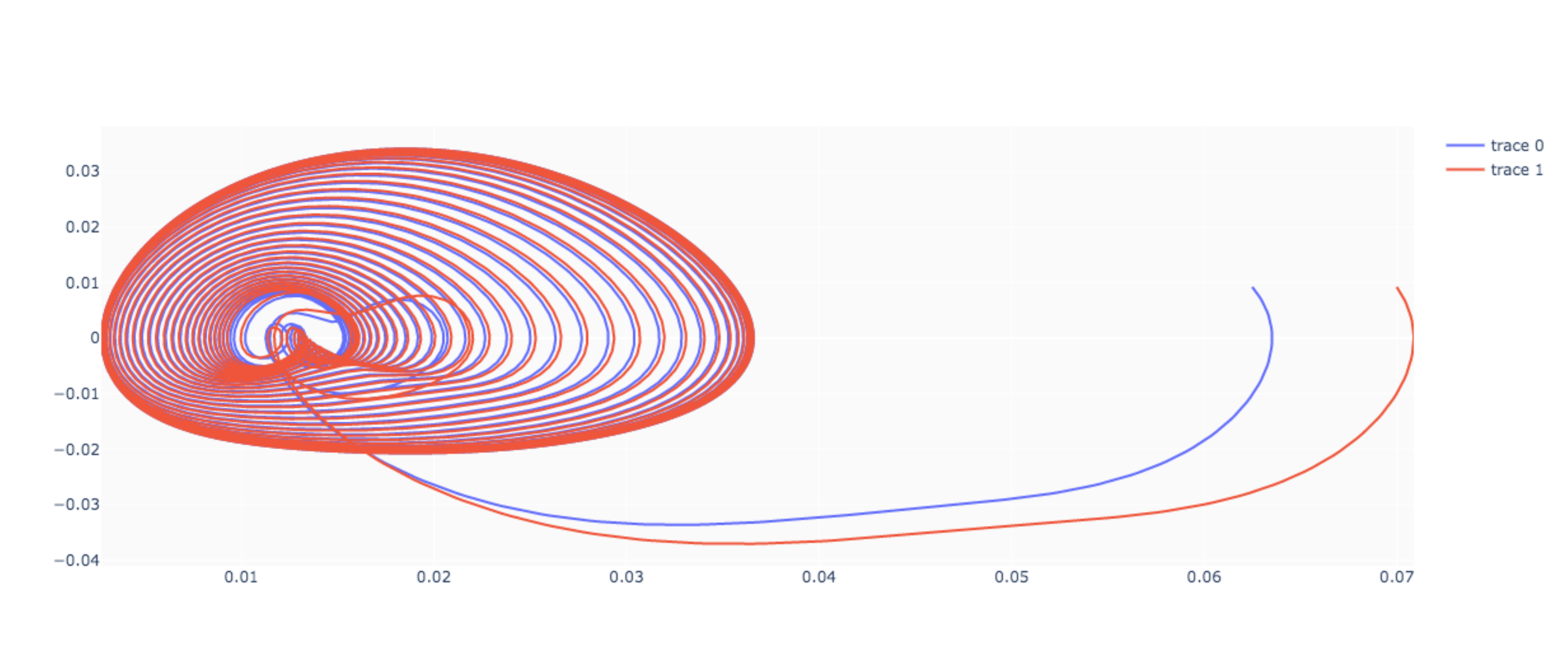}}
\caption{  
Phase portrait for two solutions of equation (\ref{system nathanson_norm non-feedback}) with delay $d=1$ and negative gain at the velocity $G_2=-8$. The asymptotically stable periodic trajectory persists under the influence of these constants.
}
\label{periodica-d-Gnegativo}
\end{figure}

Finally, in Figure \ref{periodica-d-Gnegativo} is plotted the phase portrait associated to equation (\ref{system nathanson_norm non-feedback}) with delay and negative gain at the velocity. In this case, the solutions associated with both initial conditions approach the asymptotically stable periodic trajectory established in Theorem \ref{main-theo 2} in such a way that it simulates the case of a strange attractor.


\section*{Appendix}
In this appendix we present the result of Khusainov and Yun'kova \cite{Khusainov} that provides an estimative on the delay parameter $\tau$ in the linear delay differential system
\[
(\dagger) \qquad \dot{X}(t)=AX(t)+BX(t-\tau),
\]
with  $X\in \mathbb{R}^{n}$ and $A,B\in \mathbb{M}_{n\times n}$ constant real matrices.

\begin{theorem} \label{delay size} Assume that the trivial solution of $(\dagger)$ is asymptotically stable when $\tau=0$. Let $C$ denote the real symmetric positive definite matrix satisfying 
\[
(A+B)^{T}C+C(A+B)=-I_{n},
\]
where $I_{n}$ is the $n\times n$ identity matrix. Let $\tau_{0}$ be the positive constant defined by 
\[
\tau_{0}=\Big(2\big(||A||+||B||\big)||CB||\Big)^{-1}\Big(\lambda_{\min}(C)/\lambda_{\max}(C)\Big)^{1/2},
\]
where $\lambda_{\min}(C)$ and $\lambda_{\max}(C)$ respectively denote the smallest and largest eigenvalues of $C$. Then the trivial solution of $(\dagger)$ is asymptotically stable for all $\tau < \tau_{0}.$
\end{theorem}
\begin{proof}
The proof of this result can be found in \cite{Khusainov,Gopalsamy}
\end{proof}

\bibliographystyle{elsarticle-num}
\bibliography{biblio}
\end{document}